\documentclass[12pt]{amsart}
\usepackage[T2A]{fontenc}
\usepackage[english]{babel}
\usepackage{amsaddr}
\usepackage{amsmath}
\usepackage{amssymb}
\usepackage{amsfonts}
\usepackage{epsfig}
\usepackage{srcltx}
\usepackage{subfigure}
\usepackage{cite}
\usepackage{float}
\usepackage{mathtools}
\usepackage[a4paper, mag=1000, includefoot, left=2cm, right=2cm, top=2cm, bottom=2cm, headsep=1cm, footskip=1cm]{geometry}
  \usepackage[unicode,colorlinks]{hyperref}
   \hypersetup{colorlinks=true, citecolor=blue, linkcolor=blue}

\newtheorem{theorem}{Theorem}
\newtheorem{lemma}{Lemma}

\begin{document}

\thispagestyle{empty}

\title[Resonance in isochronous systems with decaying perturbations]{Resonance in isochronous systems with decaying oscillatory perturbations}

\author[O.A. Sultanov]{Oskar A. Sultanov}

\address{
%Chebyshev Laboratory, St. Petersburg State University, 14th Line V.O., 29, Saint Petersburg 199178 Russia;\\
Institute of Mathematics, Ufa Federal Research Centre, Russian Academy of Sciences, Chernyshevsky street, 112, Ufa 450008 Russia.}
\email{oasultanov@gmail.com}

%\thanks{\it \today}

\maketitle

{\small
\begin{quote}
\noindent{\bf Abstract.} 
Non-autonomous perturbations of isochronous systems in the plane are considered. It is assumed that the intensity of perturbations decays with time, and the frequency is asymptotically constant with the limiting value satisfying a resonance condition. We discuss the emergence of attracting resonant solutions with an asymptotically constant amplitude. By combining the averaging technique and the Lyapunov function method, we show that this behaviour can occur in the phase locking and phase drifting regimes. The conditions that guarantee the existence and stability of such resonant dynamics are described.
 \medskip


%	34C15  	Nonlinear oscillations and coupled oscillators for ordinary differential equations
%	34C29  	Averaging method for ordinary differential equations
%	34E10  	Perturbations, asymptotics of solutions to ordinary differential equations
% 34D20  	Stability of solutions to ordinary differential equations

\end{quote}
}
{\small
\section*{Introduction}
An oscillatory system is called isochronous if its natural frequency does not depend on the amplitude (see, for example,~\cite{FC08}). In this paper, the isochronous systems in the plane with time-dependent oscillatory perturbations are considered. It is assumed that the intensity of perturbations decays with time, and the frequency is asymptotically constant and satisfies the resonance condition. We study long-term dynamics of the perturbed systems and discuss the emergence of attracting states with an asymptotically constant non-zero amplitude. Such effects in perturbed oscillatory systems are usually associated with the nonlinear resonance phenomenon~\cite{BVC79,SUZ88}.

Various resonant effects of periodic perturbations on isochronous systems have been considered in many papers. In particular, the boundedness of perturbed solutions was studied in~\cite{Liu09,BF09}, the emergence of unbounded solutions and the escape of solutions from the period annulus were investigated in~\cite{OR19,FF05unb,R20}. However, the nonlinear resonance phenomenon in isochronous systems with decaying oscillatory perturbations has not yet been discussed. Note that the qualitative properties of solutions to various classes of asymptotically autonomous systems were studied in~\cite{LM56,LDP74,HRT94,LRS02,KS05,MR08,OS22Non}. Decaying perturbations of oscillatory systems have been considered in several papers. In particular, the asymptotic integration of linear systems with oscillatory decaying coefficients was discussed in~\cite{MP85,PNN07,BN10}. The bifurcations near the equilibrium and the asymptotic behaviour at infinity of solutions to nonlinear systems were analyzed in~\cite{OS21DCDS,OS21JMS,OS23JMS}. To the best of the author's knowledge, the influence of such perturbations on nonlinear systems far from equilibrium and on the emergence of solutions with a steady-state amplitude has not been previously studied. This is the subject of the present paper.

The paper is organized as follows: In Section~\ref{sec1}, the statement of the problem is given. The main results are presented in Section~\ref{sec2}. The justification is provided in the subsequent sections. In particular, in Section~\ref{sec3}, we consider a change of variables that simplifies the system in the first asymptotic terms at infinity in time. Then, in Section~\ref{sec4}, under some natural assumptions on the simplified equations, we show that there are at least two possible regimes admitting asymptotically constant amplitude: the phase-locking and the phase drifting. The existence and stability of solutions with a steady-state amplitude in both regimes are studied in Section~\ref{sec5}. In Section~\ref{sec6}, the theory is illustrated with some examples.

\section{Problem Statement}\label{sec1}
Consider the system of ordinary differential equations
\begin{gather}\label{PS}
\frac{d\varrho}{d t}=f(\varrho,\varphi,S(t),t), \quad \frac{d\varphi}{dt}=\omega + g(\varrho,\varphi,S(t),t),
\end{gather}
where $\omega\in\mathbb R_+$, the functions $f(\varrho,\varphi,S,t)$ and $ g(\varrho,\varphi,S,t)$ are infinitely differentiable, defined for all $\varrho\in (0, \mathcal R]$, $(\varphi,S)\in\mathbb R^2$, $t>0$, and are $2\pi$-periodic with respect to $\varphi$ and $S$. The function $S(t)$ is smooth for $t>0$ and $S'(t)\sim s_0$ as $t\to\infty$, where $s_0$ is a positive parameter, satisfying the resonance condition: there exist coprime integers $\kappa$ and $\varkappa$ such that
\begin{gather}\label{rc}
\kappa s_0=\varkappa \omega.
\end{gather}
It is assumed that the intensity of perturbations vanishes as $t\to\infty$: for each fixed $\varrho$ and $\varphi$,
\begin{gather*}
\lim_{t\to\infty} f(\varrho,\varphi,S(t),t)=
\lim_{t\to\infty} g(\varrho,\varphi,S(t),t)=0.
\end{gather*}
Furthermore, assume that the following asymptotic expansions hold:
\begin{gather}\label{fgS}
\begin{split}
  f(\varrho,\varphi,S,t)\sim &\,\sum_{k=1}^\infty t^{-\frac{k}{q}} f_k(\varrho,\varphi,S), \\
  g(\varrho,\varphi,S,t)\sim &\,\sum_{k=1}^\infty t^{-\frac{k}{q}} f_k(\varrho,\varphi,S), \\
 S'(t)\sim &\, s_0+\sum_{k=1}^\infty t^{-\frac{k}{q}} s_k
\end{split}
\end{gather}
as $t\to\infty$ for all $\varrho\in (0, \mathcal R]$ and $(\psi,S)\in\mathbb R^2$, where the functions $f_k(\varrho,\varphi,S)$ and $g_k(\varrho,\varphi,S)$ are smooth, $2\pi$-periodic in $\varphi$ and $S$, $s_k$ are real parameters, and $q\in\mathbb Z_+$. The series in \eqref{fgS} are asymptotic as $t\to\infty$~\cite[\S 1]{MVF89}. Note that the asymptotic sequence in \eqref{fgS} may be more complex (see, for example,~\cite{PNN07}), but for simplicity, we consider the power scale. 

System \eqref{PS} is asymptotically autonomous and can be viewed as a perturbation of the system
\begin{gather}\label{ls}
\frac{d\hat\varrho}{dt}=0, \quad \frac{d\hat\varphi}{dt}=\omega
\end{gather}
describing isochronous oscillations on the plane $(x_1,x_2)=(\hat \varrho \cos\hat \varphi, - \hat\varrho\sin\hat\varphi)$ with a constant amplitude $\hat\varrho(t)\equiv \varrho_0$ and period $T=2\pi/\omega$. Hence, $\varrho(t)$ and $\varphi(t)$ play the role of the amplitude and phase of the perturbed oscillator, respectively. We study the influence of such decaying perturbations on the dynamics of the isochronous system \eqref{ls}.

Consider the example
\begin{gather}\label{Ex0}
\frac{dx_1}{dt}=x_2, \quad \frac{dx_2}{dt}=-x_1+t^{-1}\mathcal Z(x_2,S(t)),
\end{gather}
with $\mathcal Z(x,S)\equiv (a+b x+c x^2)\sin S$, $S(t)\equiv s_0 t+s_1\log t$, and $a,b,c,s_1\in\mathbb R$. Note that system \eqref{Ex0} in the variables $\varrho=\sqrt{x_1^2+x_2^2}$, $\varphi=-\arctan(x_2/x_1)$ takes the form \eqref{PS} with $q=\omega=1$, 
\begin{gather}\label{fgex0}
f(\varrho,\varphi,S,t)\equiv -t^{-1}\mathcal Z(-\varrho\sin\varphi,S)\sin\varphi, \quad 
g(\varrho,\varphi,S,t)\equiv -t^{-1}\varrho^{-1}\mathcal Z(-\varrho\sin\varphi,S)\cos\varphi.
\end{gather} 
If $a=b=c=0$, then $\varrho(t)\equiv \varrho_0$ and $\varphi(t)\equiv \omega t+\varphi_0$, where $\varrho_0$ and $\varphi_0$ are arbitrary constants. We see that if $s_0=1$, the resonance condition \eqref{rc} holds with $\kappa=\varkappa=1$. It can easily be checked that if $s_0=1$, $a\neq 0$ and $b=c=0$, then all solutions have unlimitedly growing amplitude $\varrho(t) \sim |a/2| \log t$ as $t\to\infty$. If $s_0\neq 1$, $a\neq 0$ and $b=c=0$, then $\varrho(t)\sim\varrho_0$ as $t\to\infty$, where $\varrho_0>0$ is different for different initial data. From numerical analysis of system \eqref{Ex0} it follows that similar dynamics occurs, when $b\neq 0$ and $c\neq 0$ (see Fig.~\ref{FigEx0}, a and b). If $s_0=1$, the solutions with a steady-state amplitude may appear such that $\varrho(t)\sim \varrho_\ast$ as $t\to\infty$ with some one $\varrho_\ast>0$ for different initial data (see Fig.~\ref{FigEx0}, c). We see that the emergence of these solutions is due to the resonance phenomenon. Note that in the latter case there are also solutions with unlimitedly growing amplitude. However, such modes are not discussed in this paper. See~\cite{PM23}, where such resonant modes in systems with limited nonlinearity were studied.

Thus, our goal is to find conditions that ensure the existence and stability of resonant regimes in system \eqref{PS} with an asymptotically constant amplitude $\varrho(t)\sim \varrho_\ast$ as $t\to\infty$.

\begin{figure}
\centering
\subfigure[$c=0$]{
\includegraphics[width=0.3\linewidth]{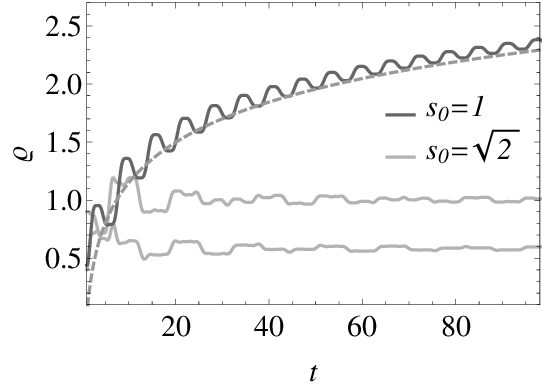}
}
\hspace{1ex}
 \subfigure[$s_0=\sqrt 2$]{
 \includegraphics[width=0.3\linewidth]{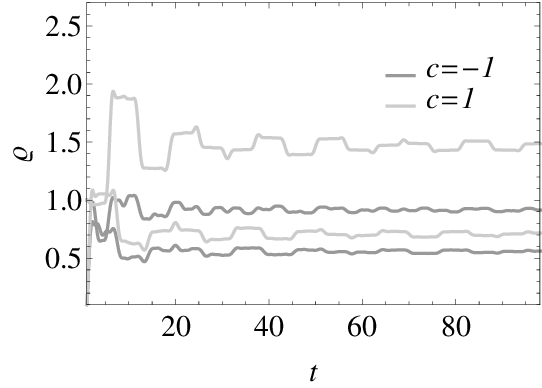}
}
\hspace{1ex}
 \subfigure[$s_0=1$]{
 \includegraphics[width=0.3\linewidth]{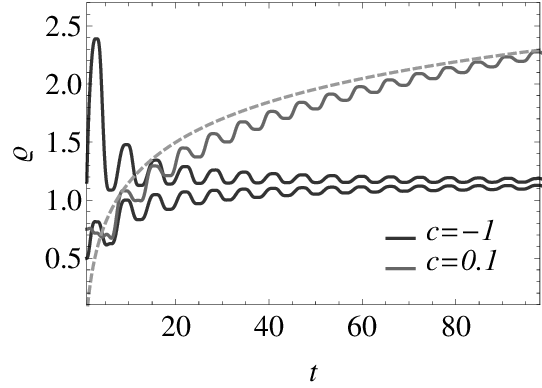}
}
\caption{\small The evolution of $\varrho(t)\equiv \sqrt{x_1^2(t)+x_2^2(t)}$ for solutions of system \eqref{Ex0} with $a=1$, $b=2$, $s_1=0$ and different values of the parameters $c$, $s_0$ and initial data. The dashed curves correspond to $\varrho(t)\equiv (1/2)\log t$.} \label{FigEx0}
\end{figure}

\section{Main results}\label{sec2}
Let us show that the perturbed system can be simplified in the first asymptotic terms as $t\to\infty$. We have the following:
\begin{theorem}\label{Th1}
Let system \eqref{PS} satisfy \eqref{rc} and \eqref{fgS}. Then, for all $N\in\mathbb Z_+$ and $\epsilon\in (0,\mathcal R/2)$ there exist $t_0>0$ and the reversible transformations $(\varrho,\varphi)\mapsto (R,\Psi)\mapsto (r,\psi)$,
\begin{gather}
\label{ch1}
	R(t)=\varrho(t), \quad 
	\Psi(t)=\varphi(t)-\frac{\kappa}{\varkappa}S(t),\\
\label{ch2}
	r(t)=R(t)+\tilde U_N(R(t),\Psi(t),t), \quad 
	\psi(t)=\Psi(t)+\tilde V_N(R(t),\Psi(t),t),
\end{gather}
where the functions $\tilde U_N(R,\Psi,t)$ and $\tilde V_N(R,\Psi,t)$ are smooth, $2\pi$-periodic in $\psi$, and satisfy the inequalities
\begin{gather*}
|\tilde U_N(R,\Psi,t)|\leq \epsilon, \quad |\tilde V_N(R,\Psi,t)|\leq \epsilon, \quad \forall\, R\in(0,\mathcal R], \quad \Psi\in\mathbb R, \quad t\geq t_0,
\end{gather*}
such that system \eqref{PS} can be transformed into
\begin{gather}\label{rpsi}
\begin{split}
&\frac{dr}{dt}=\widehat\Lambda_N(r,\psi,t)+\widetilde\Lambda_N(r,\psi,S(t),t), \quad
\frac{d\psi}{dt}=\widehat\Omega_N(r,\psi,t)+\widetilde\Omega_N(r,\psi,S(t),t),
\end{split}
\end{gather}
where 
\begin{gather}\label{LON}
\widehat\Lambda_N(r,\psi,t)\equiv  \sum_{k=1}^N t^{-\frac{k}{q}}\Lambda_k(r,\psi), \quad 
\widehat\Omega_N(r,\psi,t)\equiv \sum_{k=1}^N t^{-\frac{k}{q}}\Omega_k(r,\psi),
\end{gather}
the functions $\Lambda_k(r,\psi)$, $\Omega_k(r,\psi)$, $\widetilde\Lambda_N(r,\psi,S,t)$, $\widetilde\Omega_N(r,\psi,S,t)$ are defined for all $r\in (0, \mathcal R]$, $(\psi,S)\in\mathbb R^2$, are $2\pi$-periodic in $\psi$ and $2\pi\varkappa$-periodic in $S$. Moreover, the following estimates hold:
\begin{gather}\label{tildest}
\widetilde\Lambda_N(r,\psi,S,t)=\mathcal O(t^{-\frac{N+1}{q}}),\quad \widetilde\Omega_N(r,\psi,S,t)=\mathcal O(t^{-\frac{N+1}{q}}) 
\end{gather}
as $t\to\infty$ uniformly for all $r\in [\epsilon, \mathcal R-\epsilon]$ and $(\psi,S)\in\mathbb R^2$.
\end{theorem}

Note that the averaging transformation described in Theorem~\ref{Th1} can set some terms of \eqref{LON} to zero. Let $n\in [1,q]$ and $m\in[1,q]$ be integers such that 
\begin{gather}\label{asnm}
\begin{split}
&\Lambda_k(r,\psi)\equiv 0, \quad  k<n, \quad \Lambda_n(r,\psi)\not\equiv 0,\\
&\Omega_k(r,\psi)\equiv 0, \quad  k<m, \quad \Omega_m(r,\psi)\not\equiv 0.
\end{split}
\end{gather}
Consider the following truncated system obtained from \eqref{rpsi} by dropping the remainder terms $\widetilde \Lambda_N$ and $\widetilde \Omega_N$:
\begin{gather}\label{trsys}
\frac{d\rho}{dt}=\widehat\Lambda_N(\rho,\phi,t), \quad 
\frac{d\phi}{dt}=\widehat\Omega_N(\rho,\phi,t), \quad t\geq t_0.
\end{gather}
First, let us study possible modes for solutions to system \eqref{trsys} with a steady-state amplitude. Then, we prove the persistent of such regimes in the full system \eqref{rpsi}. 

Note that the global behaviour of solutions to asymptotically autonomous system \eqref{trsys} depends on the properties of the corresponding limiting system
\begin{gather}\label{limsys}
\frac{d\hat \rho}{dt}= t^{-\frac{n}{q}}\Lambda_n(\hat\rho,\hat\psi), \quad 
\frac{d\hat \phi}{dt}= t^{-\frac{m}{q}}\Omega_m(\hat\rho,\hat\psi), \quad t>0.
\end{gather}
Let us consider the dynamics near fixed points of system \eqref{limsys}. Assume that 
\begin{gather}\label{as1}
\exists\, \varrho_\ast\in(0,\mathcal R), \varphi_\ast\in\mathbb R: \quad \Lambda_n(\varrho_\ast,\varphi_\ast)=0, \quad \Omega_m(\varrho_\ast,\varphi_\ast)=0, \quad \mathcal D_{n,m}\neq 0,
\end{gather}
where 
\begin{gather*}
\mathcal D_{n,m}:=\det {\bf A}(\varrho_\ast,\varphi_\ast,1), \quad 
{\bf A}(\rho,\phi,t)\equiv \begin{pmatrix}t^{-\frac{n}{q}}\partial_\rho\Lambda_n(\rho,\phi) & t^{-\frac{n}{q}}\partial_\phi\Lambda_n(\rho,\phi) \\ t^{-\frac{m}{q}}\partial_\rho\Omega_m(\rho,\phi)  & t^{-\frac{m}{q}}\partial_\phi\Omega_m(\rho,\phi)
\end{pmatrix}.
\end{gather*}
Note that if $\mathcal D_{n,m}=0$, the fixed point $(\varrho_\ast,\varphi_\ast)$ is degenerate. This case may correspond to some bifurcation of the limiting system and is not discussed in this paper. The eigenvalues $\mu_1(t)$ and $\mu_2(t)$ of the matrix ${\bf A}(\varrho_\ast,\varphi_\ast,t)$ have the following form:
\begin{gather*}
\mu_1(t)\equiv t^{-\frac{n}{q}}\alpha_1(t), \quad \mu_2(t)\equiv t^{-\frac{m}{q}}\alpha_2(t), 
\end{gather*}
where  $\alpha_{j}(t)\equiv \alpha_j^0$ if $n=m$, and  $\alpha_{j}(t)\sim \alpha_j^0   $ as $t\to\infty$ if $n\neq m$. 
Define the parameters
\begin{align*}
\lambda_n=\partial_\rho \Lambda_n(\varrho_\ast,\varphi_\ast), \quad 
\nu_n=\partial_\varphi \Lambda_n(\varrho_\ast,\varphi_\ast), \quad 
\eta_m=\partial_\rho\Omega_m(\varrho_\ast,\varphi_\ast), \quad 
\omega_m=\partial_\phi\Omega_m(\varrho_\ast,\varphi_\ast).
\end{align*}
Then, we have
\begin{align*}
& \alpha_1^0=
\begin{cases}
\lambda_n, & \text{if} \quad n<m, \\
\frac{1}{2}\left(\lambda_n+\omega_m+\sqrt{(\lambda_n+\omega_m)^2-4\mathcal D_{n,m}}\right), & \text{if} \quad n=m,\\
\omega_m^{-1} \mathcal D_{n,m}, & \text{if} \quad n>m,
\end{cases}
\\ 
& \alpha_2^0=
\begin{cases}
\lambda_n^{-1} \mathcal D_{n,m}, & \text{if} \quad n<m, \\
\frac{1}{2}\left(\lambda_n+\omega_m-\sqrt{(\lambda_n+\omega_m)^2-4\mathcal D_{n,m}}\right), & \text{if} \quad n=m,\\
\omega_m, & \text{if} \quad n>m.
\end{cases}
\end{align*}
The conditions for the stability of the equilibrium $(\varrho_\ast,\varphi_\ast)$ in the limiting system \eqref{limsys} are as follows:
\begin{lemma}\label{Lem1}
Let assumption \eqref{as1} hold with $m=n$.
\begin{itemize}
\item If $\Re \alpha^0_j<0$ for all $j\in\{1,2\}$, then the equilibrium $(\varrho_\ast,\varphi_\ast)$ of system \eqref{limsys} is asymptotically stable;
\item If $\Re \alpha^0_j>0$ for some $j\in\{1,2\}$, then the equilibrium $(\varrho_\ast,\varphi_\ast)$ of system \eqref{limsys} is unstable.
\end{itemize}
\end{lemma}

\begin{lemma}\label{Lem01}
Let assumption \eqref{as1} hold with $m\neq n$.
\begin{itemize}
\item If $\alpha_1^0+\frac{n-m}{2q}\delta_{n,q}<0$ and $\alpha_2^0+\frac{m-n}{2q}\delta_{m,q}<0$, then the equilibrium $(\varrho_\ast,\varphi_\ast)$ of system \eqref{limsys} is asymptotically stable;
\item If $\alpha_1^0>0$, $\alpha_2^0>0$, and $\max\{n,m\}<q$, then the equilibrium $(\varrho_\ast,\varphi_\ast)$ of system \eqref{limsys}  is unstable.
\end{itemize}
\end{lemma}
Let us now show that similar dynamics occurs in the truncated system \eqref{trsys}. Define
\begin{gather*}
\beta_{n,m,1}=\begin{cases} 
	\Re \alpha_1^0, & \text{if } \quad n=m,\\
	\alpha_1^0+\frac{n-m}{2q}\delta_{n,q} & \text{if } \quad n\neq m,
	\end{cases}
	\quad
\beta_{n,m,2}=\begin{cases} 
	\Re \alpha_2^0, & \text{if } \quad n=m,\\
	\alpha_2^0+\frac{m-n}{2q}\delta_{m,q} & \text{if } \quad n\neq m.
	\end{cases}
\end{gather*}
Then, we have
\begin{lemma}\label{Lem2}
Let assumptions \eqref{asnm} and \eqref{as1} hold. If $\beta_{n,m,1}<0$ and $\beta_{n,m,2}<0$, then for all $N\geq \max\{n,m\}$ system \eqref{trsys} has a particular solution $\rho_\ast(t)$, $\phi_\ast(t)$ such that 
$\rho_\ast(t)=\varrho_\ast+o(1)$, $\phi_\ast(t)=\varphi_\ast+o(1)$ as $t\to\infty$. Moreover, if $\max\{n,m\}<q$, then
\begin{gather*} 
\rho_\ast(t)\sim \varrho_\ast+\sum_{k=1}^\infty t^{-\frac{k}{q}}\rho_k, \quad 
\phi_\ast(t)\sim \varphi_\ast+\sum_{k=1}^\infty t^{-\frac{k}{q}}\phi_k, \quad t\to\infty,
\end{gather*}
where $\rho_k$, $\phi_k$ are some constants.
\end{lemma}
Thus, a stable equilibrium in the limiting system corresponds to a phase-locking mode with a steady-state amplitude in the truncated system. Let us finally show that such dynamics is preserved in the full system.
\begin{theorem}\label{Th2}
Let system \eqref{PS} satisfy \eqref{rc}, \eqref{fgS}, and assumptions \eqref{asnm}, \eqref{as1} hold. 
\begin{itemize}
\item
If $\beta_{n,m,1}<0$ and $\beta_{n,m,2}<0$, then for all $\varepsilon>0$ there exist $\delta_\varepsilon\in (0,\varepsilon)$ and $t_\varepsilon> 0$ such that for all $t_s\geq t_\varepsilon$ any solution $\varrho(t)$, $\varphi(t)$ of system \eqref{PS} with initial data $|\varrho(t_s)-\rho_\ast(t_s)|+|\varphi(t_s)-\kappa S(t_s)/ \varkappa -\phi_\ast(t_s)|\leq \delta_\varepsilon$ satisfies the inequality
\begin{gather}\label{ineqrvarphi}
|\varrho(t)-\rho_\ast(t)|+\left|\varphi(t)-\frac{\kappa}{\varkappa}S(t)-\phi_\ast(t)\right|<\varepsilon
\end{gather}
for all $t>t_s$.
\item If $\beta_{n,m,1}>0$, $\beta_{n,m,2}>0$ and $\max\{n,m\}<q$, then there are $\varepsilon>0$ and $t_\ast>0$ such that for all $\delta\in (0,\varepsilon)$ the solution $\varrho(t)$, $\varphi(t)$ of system \eqref{PS} with initial data $|\varrho(t_s)-\varrho_\ast|+|\varphi(t_s)-\kappa S(t_s)/ \varkappa -\varphi_\ast|\leq \delta$ at some $t_s\geq t_\ast$ satisfies the inequality
\begin{gather}\label{ineqrvarphi2}
|\varrho(t_e)-\varrho_\ast|+\left|\varphi(t_e)-\frac{\kappa}{\varkappa}S(t_e)-\varphi_\ast\right|\geq \varepsilon
\end{gather}
at some $t_e>t_s$.
\end{itemize}
\end{theorem}

Consider now the case when instead of \eqref{as1} the following assumption holds:
\begin{gather}\label{as2}
	\begin{split}
\exists\, \varrho_\ast\in(0,\mathcal R): \quad & \Lambda_n(\varrho_\ast,\psi)\equiv 0, \quad  \ell_n(\psi):= \partial_\rho\Lambda_n(\varrho_\ast,\psi)\not\equiv 0,\\ & \Omega_n(\varrho_\ast,\psi)\neq 0 \quad \forall\,\psi\in\mathbb R.
	\end{split}
\end{gather}
In this case, stable solutions with $\varrho(t)\approx \varrho_\ast$ may appear in the phase drift regime when the phase shift tends to infinity. In particular, we have 
\begin{theorem}\label{Th3}
Let system \eqref{PS} satisfy \eqref{rc}, \eqref{fgS}, and assumptions \eqref{asnm}, \eqref{as2} hold. 
\begin{itemize}
\item
If $\ell_n(\psi)<0$, then for all $\varepsilon>0$ there exist $\delta_\varepsilon\in(0,\varepsilon)$ and $t_\varepsilon > 0$ such that for all $t_s\geq t_\varepsilon$ any solution $\varrho(t)$, $\varphi(t)$ of system \eqref{PS} with initial data $|\varrho(t_s)-\varrho_\ast|\leq \delta_\varepsilon$, $\varphi(t_s)\in\mathbb R$ satisfies the inequality
$
|\varrho(t)-\varrho_\ast|<\varepsilon
$
for all $t>t_s$. Moreover, $\varphi(t)-\kappa S(t)/\varkappa\to \infty$ as $t\to\infty$.
\item If $\ell_n(\psi)>0$, then there are $\varepsilon>0$ and $t_\ast>0$ such that for all $\delta\in (0,\varepsilon)$ the solution $\varrho(t)$, $\varphi(t)$ of system \eqref{PS} with initial data $|\varrho(t_s)-\varrho_\ast|\leq \delta$, $\varphi(t_s)\in\mathbb R$ at some $t_s\geq t_\ast$ satisfies the inequality
$ |\varrho(t_e)-\varrho_\ast|\geq \varepsilon $ at some $t_e>t_s$.
\end{itemize}
\end{theorem}

\section{Change of variables}
\label{sec3}
Substituting \eqref{ch1} into \eqref{PS} yields 
\begin{gather}\label{RP}
\frac{dR}{dt}=F(R,\Psi,S(t),t), \quad \frac{d\Psi}{dt}=G(R,\Psi,S(t),t),
\end{gather}
where
\begin{gather*}%\label{FG}
\begin{split}
&F(R,\Psi,S,t)\equiv f\left(R,\frac{\kappa}{\varkappa}S +\Psi,S,t\right), \\
&G(R,\Psi,S,t)\equiv g\left(R,\frac{\kappa}{\varkappa}S +\Psi,S,t\right)-\frac{\kappa}{\varkappa}S'(t).
\end{split}
\end{gather*}
It follows easily that 
\begin{gather}\label{FGas}
\begin{split}
F(R,\Psi,S,t)&\sim \sum_{k=1}^\infty t^{-\frac{k}{q}}F_k(R,\Psi,S), \\ 
G(R,\Psi,S,t)&\sim \sum_{k=1}^\infty t^{-\frac{k}{q}}G_k(R,\Psi,S), \quad 
t\to\infty,
\end{split}
\end{gather}
with $F_k(R,\Psi,S)\equiv f_k(R,\kappa S/\varkappa+\Psi,S)$ and $G_k(R,\Psi,S)\equiv g_k(R,\kappa S/\varkappa+\Psi,S)-\kappa s_k/\varkappa$. 
System \eqref{RP} is asymptotically autonomous with the limiting system 
\begin{gather*}
\frac{d\hat R}{dt}=0, \quad \frac{d\hat\Psi}{dt}=0, \quad \frac{d\hat S}{dt}=s_0.
\end{gather*}
In this case, the function $S(t)$ can be referred as a fast variable as $t\to\infty$ and system \eqref{RP} can be simplified by averaging the equations with respect to $S$ (see, for example,~\cite{BM61}). Consider the following near-identity transformation:
\begin{gather*}%\label{uv}
\begin{split}
	U_N(R,\Psi,S,t)&=R+\sum_{k=1}^N t^{-\frac{k}{q}} u_k(R,\Psi,S), \\
	V_N(R,\Psi,S,t)&=\Psi+\sum_{k=1}^N t^{-\frac{k}{q}} v_k(R,\Psi,S)
\end{split}
\end{gather*}
for all $R\in (0,\mathcal R]$, $(\Psi,S)\in\mathbb R^2$ and $t>0$ with some integer $N\geq 1$. The coefficients $u_k(R,\Psi,S)$ and $v_k(R,\Psi,S)$ are assumed to be periodic in $\Psi$ and $S$. These coefficients are sought in such a way that the system, written in the new variables 
\begin{gather}\label{chrp}
r(t)\equiv U_N(R(t),\Psi(t),S(t),t), \quad  
\psi(t)\equiv V_N(R(t),\Psi(t),S(t),t),
\end{gather}
takes the form \eqref{rpsi}. From \eqref{fgS} and \eqref{FGas} it follows that 
\begin{gather}\label{duv}
\begin{split}
\frac{d}{dt}\begin{pmatrix} U_N  \\ V_N \end{pmatrix}:= &\left(\frac{dR}{dt}\partial_R+\frac{d\Psi}{dt}\partial_\Psi+\frac{dS}{dt}\partial_S+\partial_t\right)\begin{pmatrix} U_N  \\ V_N \end{pmatrix}\\
\sim &
\sum_{k=1}^\infty t^{-\frac{k}{q}} \left\{\begin{pmatrix} F_k \\ G_k \end{pmatrix}+s_0\partial_S \begin{pmatrix} u_k \\ v_k \end{pmatrix}-\frac{k-q}{q} \begin{pmatrix} u_{k-q} \\ v_{k-q} \end{pmatrix} \right\} \\ &
 +\sum_{k=2}^\infty t^{-\frac{k}{q}} \sum_{j=1}^{k-1}( F_j\partial_R + G_j\partial_\Psi+s_{j} \partial_S)\begin{pmatrix} u_{k-j} \\ v_{k-j} \end{pmatrix}
\end{split}
\end{gather}
as $t\to\infty$, where it is assumed that $u_k(R,\Psi,S)\equiv v_k(R,\Psi,S)\equiv 0$ for $k\leq 0$ and $k>N$. Comparing the terms of the same powers of $t^{-1/q}$ in \eqref{rpsi} and \eqref{duv}, we get the following system:
\begin{gather}\label{ukvk}
s_0\partial_S \begin{pmatrix} u_k \\ v_k \end{pmatrix}=\begin{pmatrix} \Lambda_k(R,\Psi)-F_k(R,\Psi,S) +\tilde F_k(R,\Psi,S)\\ \Omega_k(R,\Psi)-G_k(R,\Psi,S)+\tilde G_k(R,\Psi,S) \end{pmatrix}, \quad k=1,\dots, N,
\end{gather} 
where the functions $\tilde F_k(R,\Psi,S)$ and $\tilde G_k(R,\Psi,S)$ are expressed through $\{u_{j}, v_{j},\Lambda_{j}, \Omega_{j}\}_{j=1}^{k-1}$. In particular,
\begin{align*}
	\begin{pmatrix} \tilde F_1 \\ \tilde G_1 \end{pmatrix} \equiv  
		& \begin{pmatrix} 0\\ 0 \end{pmatrix},
		\\
	\begin{pmatrix}\tilde F_2 \\ \tilde G_2\end{pmatrix} \equiv 
		& -(F_1\partial_R + G_1\partial_\Psi +s_{1}\partial_S)\begin{pmatrix} u_{1} \\ v_{1} \end{pmatrix} + \frac{2-q}{q} \begin{pmatrix} u_{2-q} \\ v_{2-q} \end{pmatrix} + (u_1\partial_R+v_1\partial_\Psi) \begin{pmatrix} \Lambda_1\\ \Omega_1 \end{pmatrix} 
			\\
	\begin{pmatrix} \tilde F_3 \\ \tilde G_3 \end{pmatrix} \equiv 
		& -\sum_{j=1}^2 ( F_j\partial_R + G_j\partial_\Psi+s_{j}\partial_S)\begin{pmatrix} u_{3-j} \\ v_{3-j} \end{pmatrix} + \frac{3-q}{q} \begin{pmatrix} u_{3-q} \\ v_{3-q} \end{pmatrix}
		\\
		& +\sum_{i+j=3}(u_i\partial_R+v_i\partial_\Psi) \begin{pmatrix} \Lambda_j\\ \Omega_j \end{pmatrix} + \frac{1}{2}\left(u_1^2\partial_R^2+2u_1v_1 \partial_R\partial_\Psi+v_1^2 \partial_\Psi^2\right)\begin{pmatrix} \Lambda_1\\ \Omega_1 \end{pmatrix}.
\end{align*}
Let us take
\begin{align*}
\Lambda_k(R,\Psi)&\equiv \langle F_k(R,\Psi,S)-\tilde F_k(R,\Psi,S)\rangle_{\varkappa S},\\
\Omega_k(R,\Psi)&\equiv \langle G_k(R,\Psi,S)-\tilde G_k(R,\Psi,S)\rangle_{\varkappa S}.
\end{align*}
Then, the right-hand side of \eqref{ukvk} is periodic in $S$ and has zero mean.  It follows that system \eqref{ukvk} is solvable in the class of functions that are $2\pi\varkappa$-periodic in $S$. We see that $u_k(R,\Psi,S)$, $v_k(R,\Psi,S)$, $\Lambda_k(R,\Psi)$, $\Omega_k(R,\Psi)$ are $2\pi$-periodic in $\Psi$. Thus, for all $\epsilon\in (0,\mathcal R/2)$ there exists $t_0\geq 1$ such that
\begin{gather*}
|U_N(R,\Psi,S,t)-R|< \epsilon, \quad 
|V_N(R,\Psi,S,t)-\Psi|< \epsilon, \quad 
|D(R,\Psi,S,t)-1|< \epsilon
\end{gather*}
for all $R\in(0,\mathcal R]$, $(\Psi,S)\in\mathbb R^2$ and $t\geq t_0$, where
\begin{gather*}
D(R,\Psi,S,t):={\hbox{\rm det}} \begin{vmatrix} \partial_R U_N(R,\Psi,S,t) & \partial_\Psi U_N(R,\Psi,S,t) \\ \partial_R V_N(R,\Psi,S,t) & \partial_\Psi V_N(R,\Psi,S,t) \end{vmatrix}.
\end{gather*}
 Hence, \eqref{chrp} is invertible and there exists the reverse transformation $R= u(r,\psi,t)$, $\Psi=v(r,\psi,t)$  such that $0<u(r,\psi,t)<\mathcal R$ for all $r\in [\epsilon,\mathcal R-\epsilon]$, $\psi\in\mathbb R$ and $t\geq t_0$.
Define
\begin{align*}
\begin{pmatrix}
\tilde \Lambda_N(r,\psi,S,t)\\
\tilde \Omega_N(r,\psi,S,t)
\end{pmatrix}
\equiv &
\frac{d}{dt}\begin{pmatrix} U_N  \\ V_N \end{pmatrix}\Big|_{R=u(r, \psi, t), \Psi=v(r,\psi,t)}  - \sum_{k=1}^N t^{-\frac{k}{q}} \begin{pmatrix}
\Lambda_k(r,\psi)\\
\Omega_k(r,\psi)
\end{pmatrix}.
\end{align*}
Then, the estimates \eqref{tildest} hold.

Thus, we obtain the proof of Theorem~\ref{Th1} with $\tilde U_N(R,\Psi,t)\equiv  U_N(R,\Psi,S(t),t)-R$, $\tilde V_N(R,\Psi,t)\equiv  V_N(R,\Psi,S(t),t)-\Psi$.

\section{Analysis of the truncated systems}\label{sec4}

\begin{proof}[Proof of Lemma~\ref{Lem1}]
Substituting $\rho(t)=\varrho_\ast+y_1(t)$, $\phi(t)=\varphi_\ast+y_2(t)$ into \eqref{limsys}, we obtain
\begin{gather}\label{ysys}
 \frac{d{\bf y}}{dt}= {\bf A}(\varrho_\ast,\varphi_\ast,t) {\bf y}+{\bf h}({\bf y},t), 
\end{gather}
where ${\bf y}=(y_1,y_2)^T$, ${\bf h}({\bf y},t)\equiv {\bf a}({\bf y},t)-{\bf A}(\varrho_\ast,\varphi_\ast,t) {\bf y}$, and
\begin{gather*}
{\bf a}({\bf y},t)\equiv \begin{pmatrix} t^{-\frac{n}{q}}\Lambda_n(\varrho_\ast+y_1,\varphi_\ast+y_2)\\ t^{-\frac{m}{q}}\Omega_m(\varrho_\ast+y_1,\varphi_\ast+y_2)\end{pmatrix}.
\end{gather*}

Let matrix ${\bf A}(\varrho_\ast,\varphi_\ast,1)$ is non-defective. Then, there exists an invertible matrix ${\bf T}_0$ such that 
\begin{gather}\label{T0}
{\bf T}^{-1}_0{\bf A}(\varrho_\ast,\varphi_\ast,t) {\bf T}_0 = t^{-\frac nq}{\bf D}_0, \quad {\bf D}_0 ={\hbox{\rm diag}}\{\alpha_1^0,\alpha_2^0\}.
\end{gather}
Hence, substituting ${\bf y}(t)={\bf T}_0 {\bf z}(t)$ into \eqref{ysys}, we obtain
\begin{gather}\label{zsys}
t^{ \frac{n}{q}} \frac{d{\bf z}}{d\tau}=  {\bf D}_0 {\bf z}+ {\bf b}_0({\bf z}), \quad {\bf b}_0({\bf z}) \equiv t^{\frac{n}{q}}{\bf T}_0^{-1}{\bf h}({\bf T}_0 {\bf z},t)=\mathcal O(|{\bf z}|^2), \quad |{\bf z}|=\sqrt{|z_1|^2+|z_2|^2}\to 0.
\end{gather}
Note that the matrix ${\bf D}_0$ and the vector ${\bf z}=(z_1,z_2)^T$ may have complex elements.  If $\Re \alpha_{j}^0<0$ for all $j\in\{1,2\}$, consider 
\begin{gather}\label{L0}
L({\bf z})\equiv ({\bf z},\overline{\bf z})\equiv |z_1|^2+|z_2|^2
\end{gather} 
as a Lyapunov function candidate for system \eqref{zsys}, where $\overline{\bf z}=(\overline{z}_1, \overline{z}_2)^T$ is the conjugate of a complex vector $\overline{\bf z}$. The derivative of $L({\bf z})$ along the  trajectories of system \eqref{zsys} satisfies the following:
\begin{gather*}
\frac{dL}{dt}\Big|_\eqref{zsys}= t^{-\frac{n}{q}} \left(2\sum_{j=1}^2 \Re \alpha_j^0 |z_j|^2 + \mathcal O(|{\bf z}|^3)\right), \quad |{\bf z}|\to 0.
\end{gather*}
Hence, there exists $\Delta_0>0$ such that 
\begin{gather}\label{Lineq}
\frac{dL}{dt}\Big|_\eqref{zsys} \leq  -\gamma t^{-\frac{n}{q}}L
\end{gather}
for all $|{\bf z}|\leq \Delta_0$ and $t>0$ with $\gamma=\min\{|\Re \alpha_1^0|,|\Re \alpha_2^0|\}$. Therefore, for all $\epsilon\in (0,\Delta_0)$ there exists $\delta\in (0,\epsilon)$ such that any solution ${\bf z}(t)$ of system \eqref{zsys} with initial data $|{\bf z}(t_s)|\leq \delta$, $t_s>0$ cannot exit from the domain $\{|{\bf z}|\leq \epsilon\}$ as $t>t_s$. Moreover, integrating \eqref{Lineq}, we get
\begin{equation*}%\label{estz}
\begin{array}{rlrl}
|{\bf z}(t)|&\leq |{\bf z}(t_s)| \exp \left(-\frac{\gamma q}{n-q} (t^{1-\frac{n}{q}}-t_s^{1-\frac{n}{q}})\right), & \text{if} &  n<q,\\
 |{\bf z}(t)|&\leq |{\bf z}(t_s)|  \left(\frac{t}{t_0}\right)^{-\gamma}, & \text{if} &    n =q,
\end{array}
\end{equation*}
as $t>t_s$. Since $|{\bf y}|\leq B_0 |{\bf z}|$ with some $B_0={\hbox{\rm const}}>0$, we see that the solution ${\bf y}(t)\equiv 0$ is asymptotically stable in system \eqref{ysys}. 
If $\Re \alpha_{1}^0>0$, consider $J({\bf z})\equiv  |z_1|^2+({\hbox{\rm sgn} \Re \alpha_2^0})|z_2|^2$ as a Chetaev function candidate for system \eqref{zsys}. It can easily be checked that 
\begin{gather*} 
\frac{dJ}{dt}\Big|_\eqref{zsys} = t^{-\frac{n}{q}} \left( 2 \Re \alpha_1^0 |  z_1|^2+2 |\Re \alpha_2^0| |  z_2|^2 + \mathcal O(|  {\bf z}|^3)\right), \quad | {\bf z}|\to 0.
\end{gather*}
Hence, there exists $\Delta_1>0$ such that 
\begin{gather*}%\label{Jest} 
\frac{dJ}{dt}\Big|_\eqref{zsys} \geq \gamma t^{-\frac{n}{q}} |{\bf z}|^2\geq 0
\end{gather*}
for all ${\bf z}\in\{|  {\bf z}|\leq \Delta_1: J({\bf z})\geq 0\}$ and $t>0$. It follows that the solution ${\bf z}(t)\equiv 0$ is unstable (see, for example,~\cite[Theorem~4.3]{HKH}).  Since $|{\bf y}|\geq A_0|{\bf z}|$ with some $A_0\in (0,B_0)$, we see that the solution ${\bf y}(t)\equiv 0$ is unstable in system \eqref{ysys}. 

Let matrix ${\bf A}(\varrho_\ast,\varphi_\ast,1)$ is defective. Then, $\alpha_1^0=\alpha_2^0=\alpha_0\in \mathbb R$ and there exists an invertible matrix $\tilde{\bf D}_0$ such that
\begin{gather*}
\tilde {\bf T}^{-1}_0{\bf A}(\varrho_\ast,\varphi_\ast,t) \tilde {\bf T}_0 = t^{-\frac nq}\tilde {\bf D}_0, \quad \tilde{\bf D}_0 =\begin{pmatrix} a_0 & 1 \\ 0 &  \alpha_0\end{pmatrix}.
\end{gather*}
Substituting ${\bf y}(t)=\tilde {\bf T}_0 \tilde {\bf z}(t)$ into \eqref{ysys}, we obtain
\begin{gather}\label{z2sys}
t^{ \frac{n}{q}} \frac{d\tilde {\bf z}}{d\tau}=  \tilde {\bf D_0} \tilde {\bf z}+ \tilde{\bf b}_0(\tilde {\bf z}), \quad 
\tilde{\bf b}_0(\tilde {\bf z}) \equiv t^{\frac{n}{q}}\tilde {\bf D}_0^{-1}{\bf h}(\tilde {\bf D}_0 \tilde {\bf z},t)=\mathcal O(|\tilde {\bf z}|^2), \quad |\tilde {\bf z}|=\sqrt{\tilde z_1^2+\tilde z_2^2}\to 0.
\end{gather}
Consider $\tilde L(\tilde{z}_1,\tilde z_2)\equiv \tilde z_1^2+\tilde z_2^2/a_0^2$ as a Lyapunov function candidate. Then, its derivative satisfies
\begin{gather*} 
\frac{d\tilde L}{dt}\Big|_\eqref{z2sys} = t^{-\frac{n}{q}} \left( 2 \alpha_0  \tilde z_1^2+2 \tilde z_1 \tilde z_2 +  \frac{\tilde z_2^2}{\alpha_0} + \mathcal O(|  \tilde {\bf z}|^3)\right), \quad | \tilde{\bf z}|\to 0
\end{gather*}
for all $t>0$. It can easily be seen that there exists $\Delta_2>0$ such that 
\begin{eqnarray*}
 \frac{d\tilde L}{dt}\Big|_\eqref{z2sys} \leq  -\tilde \gamma t^{-\frac{n}{q}}\tilde L,  & \quad & \text{if} \quad   \alpha_0<0;\\
 \frac{d\tilde L}{dt}\Big|_\eqref{z2sys} \geq  \tilde \gamma t^{-\frac{n}{q}}\tilde L, & \quad & \text{if} \quad  \alpha_0>0
\end{eqnarray*}
for all $|\tilde {\bf z}|\leq \Delta_2$ and $t>0$ with $\tilde \gamma=|\alpha_0|/2$. Hence, the trivial solution of system \eqref{ysys} is asymptotically stable if $\alpha_0<0$ and unstable if $\alpha_0>0$. 

Returning to the variables $(\rho,\phi)$, we obtain the stability or instability of the equilibrium $(\varrho_\ast,\varphi_\ast)$ in system \eqref{limsys}.

\end{proof}

\begin{proof}[Proof of Lemma~\ref{Lem01}]
For definiteness let us assume that $n<m$ (The proof in the case $n>m$ is similar). Then, $\alpha_1^0=\lambda_n$ and $\alpha_2^0=\lambda_n^{-1}\mathcal D_{n,m}$. Define  $\tilde\omega_m=\omega_m+\delta_{m,q}(m-n)/(2q)$, $\tilde \alpha_2^0=\alpha_2^0+\delta_{m,q}(m-n)/(2q)$ and consider 
\begin{gather}\label{LFunc}
L(y_1,y_2,t)=C_1 y_1^2+t^{\frac{m-n}{q}} y_2^2 + C_2 y_1y_2
\end{gather}
with some parameters $C_1$ and $C_2$ as a Lyapunov function candidate for system \eqref{ysys}. The derivative of $L(y_1,y_2,t)$ along the 
trajectories of system \eqref{ysys} is given by 
\begin{align*}
\frac{dL}{dt}\Big|_{\eqref{ysys}}=& \, t^{-\frac{n}{q}}\left\{
2C_1 \lambda_n y_1^2+(2 \tilde\omega_m + C_2 \nu_n)y_2^2 + (2C_1\nu_n+2 \eta_m+C_2 \lambda_n)+\mathcal O(|{\bf y}|^3)+\mathcal O(t^{-\frac{1}{q}})\mathcal O(|{\bf y}|^2)\right\}
\end{align*}
as $|{\bf y}|\to 0$ and $t\to\infty$. We take  
\begin{gather}\label{C1C2}
C_1=
\begin{cases}
\lambda_n \tilde\omega_m, & \text{if} \quad \nu_n=0,\\
\displaystyle \frac{\lambda_n \tilde \alpha_2^0}{2\nu_n^2}, & \text{if} \quad \nu_n\neq 0,
\end{cases}
\quad  
C_2=-\frac{2C_1\nu_n+2\eta_m}{\lambda_n}.
\end{gather}
Then, the following estimate holds:
\begin{gather*}
\frac{dL}{dt}\Big|_{\eqref{ysys}}= 2C_1 \lambda_n t^{-\frac{n}{q}} \left\{y_1^2+(\nu_n^2+\delta_{\nu_n,0})\lambda_n^{-2} y_2^2+\mathcal O(|{\bf y}|^3)+\mathcal O(t^{-\frac{1}{q}})\mathcal O(|{\bf y}|^2)\right\}
\end{gather*}
as $|{\bf y}|\to 0$ and $t\to\infty$. 
It can easily be checked that $C_1>0$ if $\alpha_1^0 \tilde \alpha_2^0>0$. Therefore, there exist $\Delta_1>0$ and $t_1>0$ such that
\begin{equation*}
\begin{array}{cl}
B_- |{\bf y}|^2\leq L(y_1,y_2,t)\leq B_+ t^{\frac{m-n}{q}}|{\bf y}|^2,&\\
\displaystyle \frac{dL}{dt}\Big|_{\eqref{ysys}}\leq -\gamma t^{-\frac{m}{q}} L, & \text{if} \quad \alpha_1^0<0, \quad \tilde \alpha_2^0<0,\\
\displaystyle \frac{dL}{dt}\Big|_{\eqref{ysys}}\geq  \gamma t^{-\frac{m}{q}} L, & \text{if} \quad \alpha_1^0>0, \quad \tilde \alpha_2^0>0
\end{array}
\end{equation*}
for all $|{\bf y}|\leq \Delta_1$ and $t\geq t_1$ with some constants $B_\pm>0$ and $\gamma=C_1|\alpha_1^0|\min\{1,(\nu_n^2+\delta_{\nu_n,0})\lambda_n^{-2}\}/B_+>0$.
It follows that the trivial solution of system \eqref{ysys} is asymptotically stable if $\alpha_1^0<0$ and $\tilde \alpha_2^0<0$. 
If $\alpha_1^0>0$, $\alpha_2^0>0$ and $m<q$, then 
\begin{gather*}
|{\bf y}(t)| \geq B_1 t^{-\frac{m-n}{2q}} \exp \left(\frac{\gamma q}{2(q-m)}t^{1-\frac{m}{q}}\right), \quad t\geq t_1
\end{gather*}
with some constant $B_1>0$. In this case, the solution is unstable. Returning to the variables $(\rho,\phi)$, we obtain the stability or instability of the equilibrium $(\varrho_\ast,\varphi_\ast)$ in system \eqref{limsys}.
\end{proof}

\begin{proof}[Proof of Lemma~\ref{Lem2}]
Consider the case $n\neq m$. The proof in the case $n=m$ is similar. 
For definiteness let us assume that $n<m$. Substituting $\rho(t)=\varrho_\ast+y_1(t)$, $\phi(t)=\varphi_\ast+y_2(t)$ into \eqref{trsys}, we obtain
\begin{gather}\label{ysys2}
 \frac{dy_1}{dt}= \mathcal F(y_1,y_2,t), \quad 
\frac{dy_2}{dt}= \mathcal G(y_1,y_2,t), 
\end{gather}
where
\begin{gather*}
\mathcal F(y_1,y_2,t)=\widehat\Lambda_N(\varrho_\ast+y_1,\varphi_\ast+y_2,t), \quad 
\mathcal G(y_1,y_2,t)=\widehat\Omega_N(\varrho_\ast+y_1,\varphi_\ast+y_2,t).
\end{gather*}

{\bf 1}. If $n< m<q$, the asymptotic solution for system \eqref{ysys2} is constructed in the form
\begin{gather}\label{xizeta}
y_1(t)=\sum_{k=1}^\infty t^{-\frac{k}{q}} \xi_k, \quad 
y_2(t)=\sum_{k=1}^\infty t^{-\frac{k}{q}} \zeta_k, \quad t\to\infty,
\end{gather}
where $\xi_k$ and $\zeta_k$ are constant coefficients. Substituting these series into \eqref{ysys} and grouping the terms of the same power of $t$ yield the following chain of equations for the coefficients $\xi_k$ and $\zeta_k$:
\begin{gather}\label{xikzetak}
-{\bf A}(\varrho_\ast,\varphi_\ast,1) \begin{pmatrix}\xi_k \\ \zeta_k \end{pmatrix} = 
\begin{pmatrix} \mathfrak F_k(\xi_1,\zeta_1,\dots,\xi_{k-1},\zeta_{k-1}) \\
\mathfrak G_k(\xi_1,\zeta_1,\dots,\xi_{k-1},\zeta_{k-1}) 
\end{pmatrix}, \quad k\geq 1,
\end{gather}
where the right-hand sides  are certain polynomial functions of $\xi_1,\zeta_1,\dots,\xi_{k-1},\zeta_{k-1}$.
For example, $\mathfrak F_1= \Lambda_{n+1}(\varrho_\ast,\varphi_\ast)$, $\mathfrak G_1= \Omega_{m+1}(\varrho_\ast,\varphi_\ast)$, $\mathfrak F_2= \Lambda_{n+2}(\varrho_\ast,\varphi_\ast)+(\xi_1 \partial_\rho +\zeta_1 \partial_\phi) \Lambda_{n+1}(\varrho_\ast,\varphi_\ast)+(\xi_1^2\partial_\rho^2+2\xi_1\zeta_1 \partial_\rho \partial_\phi+\zeta_1^2\partial_\phi^2)\Lambda_n(\varrho_\ast,\varphi_\ast)/2$, and $\mathfrak G_2= \Omega_{m}(\varrho_\ast,\varphi_\ast)+(\xi_1 \partial_\rho +\zeta_1 \partial_\phi) \Omega_{m+1}(\varrho_\ast,\varphi_\ast)+(\xi_1^2\partial_\rho^2+2\xi_1\zeta_1 \partial_\rho \partial_\phi+\zeta_1^2\partial_\phi^2)\Omega_m(\varrho_\ast,\varphi_\ast)/2$.  Since $\mathcal D_{n,m}\neq 0$, system \eqref{xikzetak} is solvable. To prove the existence of solutions to system \eqref{ysys2} with asymptotic behaviour \eqref{xizeta}, we construct a suitable Lyapunov function. This method was used for similar cases in~\cite{LK15}. Define the functions
\begin{gather*}
\Xi_M(t)\equiv \sum_{k=1}^M t^{-\frac{k}{q}}\xi_k, \quad 
Z_M(t)\equiv \sum_{k=1}^M t^{-\frac{k}{q}}\zeta_k, \quad M\in\mathbb Z_+.
\end{gather*}
Note that
\begin{gather}
\label{R1R2}
\begin{split}
&\mathfrak R_1(t)\equiv \frac{d\Xi_M(t)}{dt}-\mathcal F(\Xi_M(t),Z_M(t),t)=\mathcal O\left(t^{-\frac{n+M+1}{q}}\right), \\
&\mathfrak R_2(t)\equiv\frac{dZ_M(t)}{dt}-\mathcal G(\Xi_M(t),Z_M(t),t)=\mathcal O\left(t^{-\frac{m+M+1}{q}}\right), \quad t\to\infty.
\end{split}
\end{gather}
Substituting $y_1(t)=\Xi_M(t)+t^{-\frac{M}{2q}}z_1(t)$, $y_2(t)=Z_M(t)+t^{-\frac{M}{2q}}z_2(t)$ into \eqref{ysys2}, we obtain 
\begin{gather}\label{z1z2sys}
\frac{dz_1}{dt}=\tilde {\mathcal F}(z_1,z_2,t)-t^{\frac{M}{2q}}\mathfrak R_1(t), \quad 
\frac{dz_2}{dt}=\tilde {\mathcal G}(z_1,z_2,t)-t^{\frac{M}{2q}}\mathfrak R_2(t),
\end{gather}
where 
\begin{align*}
	\tilde {\mathcal F}(z_1,z_2,t)\equiv & t^{\frac{M}{2q}}\left(\mathcal F(\Xi_M(t)+t^{-\frac{M}{2q}}z_1,Z_M(t)+t^{-\frac{M}{2q}}z_2,t)-\mathcal F(\Xi_M(t),Z_M(t),t)\right)+\frac{M}{2q}t^{-1}z_1, \\
	\tilde {\mathcal G}(z_1,z_2,t)\equiv &  t^{\frac{M}{2q}}\left(\mathcal G(\Xi_M(t)+t^{-\frac{M}{2q}}z_1,Z_M(t)+t^{-\frac{M}{2q}}z_2,t)-\mathcal G(\Xi_M(t),Z_M(t),t)\right)+\frac{M}{2q}t^{-1}z_2.
\end{align*}
It can easily be checked that
\begin{align*}
	 & \tilde {\mathcal F}(z_1,z_2,t)=t^{-\frac{n}{q}}\left(\lambda_n z_1+\nu_n z_2+\mathcal O(t^{-\frac{1}{q}})\mathcal O(|{\bf z}|)\right), \\ 
	& \tilde {\mathcal G}(z_1,z_2,t)=t^{-\frac{m}{q}}\left(\eta_m z_1+\omega_m z_2+\mathcal O(t^{-\frac{1}{q}})\mathcal O(|{\bf z}|)\right) 
\end{align*}
as $|{\bf z}|=\sqrt{z_1^2+z_2^2}\to 0$ and $t\to\infty$. Consider $V(z_1,z_2,t)\equiv L(z_1,z_2,t)$ as a Lyapunov function candidate for system \eqref{z1z2sys}, where
$L(z_1,z_2,t)$ is defined by \eqref{LFunc} with 
\begin{gather*}
C_1=
\begin{cases}
\displaystyle \beta_{n,m,1} \omega_m, & \text{if} \quad \nu_n=0,\\
\displaystyle \frac{\beta_{n,m,1} \beta_{n,m,2}}{2\nu_n^2}, & \text{if} \quad \nu_n\neq 0,
\end{cases}
\quad  
C_2=-\frac{2C_1\nu_n+2\eta_m}{\beta_{n,m,1}}.
\end{gather*}
Note that $C_1>0$.
The derivative of $V(z_1,z_2,t)$ along the trajectories of system \eqref{z1z2sys} satisfies
\begin{gather*}
\frac{dV}{dt}\Big|_{\eqref{z1z2sys}}= 2C_1 \beta_{n,m,1} t^{-\frac{n}{q}} \left\{z_1^2+(\nu_n^2+\delta_{\nu_n,0})\beta_{n,m,1}^{-2}z_2^2 +\mathcal O(t^{-\frac{1}{q}})\mathcal O(|{\bf z}|^2)\right\}+\mathcal O(t^{-\frac{2n+M+2}{2q}})\mathcal O(|{\bf z}|)
\end{gather*}
as $|{\bf z}|\to 0$ and $t\to\infty$. It follows that there exist $\Delta_1>0$ and $t_1\geq t_0$ such that
\begin{gather*} 
B_- |{\bf z}|^2\leq V(z_1,z_2,t)\leq B_+ t^{\frac{m-n}{q}}|{\bf z}|^2,\quad
\frac{dV}{dt}\Big|_{\eqref{z1z2sys}}\leq t^{-\frac{n}{q}}\left(-\gamma |{\bf z}|^2+ t^{-\frac{M+2}{2q}} B_\ast |{\bf z}|\right)
\end{gather*}
for all $|{\bf z}|\leq \Delta_1$ and $t\geq t_1$ with some constants $B_\pm>0$, $B_\ast>0$ and $\gamma=C_1|\beta_{n,m,1}|\min\{1,(\nu_n^2+\delta_{\nu_n,0})\beta_{n,m,1}^{-2}\}>0$. Taking $M\geq m-n-1$ ensures that for all $\varepsilon\in (0,\Delta_1)$ there exist
\begin{gather*}
\delta_\varepsilon=\frac{2B_\ast }{\gamma}t_\varepsilon^{-\frac{M+2-(m-n)}{2q}}, \quad
t_\varepsilon=\max\left\{1, t_1, \left(\frac{4B_\ast}{\gamma \varepsilon}\sqrt{\frac{B_+}{B_-}}\right)^{\frac{2q}{M+2}}\right\}
\end{gather*}
such that
\begin{gather*}
\frac{dV}{dt}\Big|_{\eqref{z1z2sys}}\leq t^{-\frac{n}{q}}|{\bf z}|^2\left(-\gamma+t_\varepsilon^{-\frac{M+2-(m-n)}{2q}}\frac{B_\ast}{\delta_\varepsilon}\right)\leq -t^{-\frac{n}{q}}\frac{\gamma |{\bf z}|^2}{2} <0
\end{gather*}
for all $\delta_\varepsilon  t^{-\frac{m-n}{2q}}\leq |{\bf z}|\leq \varepsilon$ and $t\geq t_\varepsilon$. Combining this with the inequalities
\begin{gather}\label{Vineq}
\sup_{|{\bf z}|\leq \delta_\varepsilon t^{-\frac{m-n}{2q}}, t\geq t_\varepsilon} V(z_1,z_2,t)\leq B_+ \delta_\varepsilon^2 < B_-\varepsilon^2=\inf_{|{\bf z}|=\varepsilon, t\geq t_\varepsilon } V(z_1,z_2,t)
\end{gather}
we see that any solution $(z_1(t),z_2(t))$ of system \eqref{z1z2sys} with initial data $|{\bf z}(t_s)|\leq \delta_\varepsilon t_s^{-(m-n)/(2q)}$, $t_s\geq t_\varepsilon$ cannot exit from the domain $|{\bf z}|\leq\varepsilon$ as $t> t_s$. Hence, for all $M\geq m-n-1$ there exists a solution of system \eqref{ysys2} such that $y_1(t)=\Xi_M(t)+\mathcal O(t^{-M/(2q)})$, $y_2(t)=Z_M(t)+\mathcal O(t^{-M/(2q)})$ as $t\to\infty$. 

{\bf 2}. If $n<m=q$, the asymptotics is constructed in the following form
\begin{gather*}%\label{xizeta}
y_1(t)=\sum_{k=1}^\infty t^{-\frac{k}{q}} \xi_k(\log t), \quad 
y_2(t)=\sum_{k=1}^\infty t^{-\frac{k}{q}} \zeta_k(\log t), \quad t\to\infty,
\end{gather*}
where $\xi_k(\tau)$ and $\zeta_k(\tau)$ are some polynomials in $\tau$.
Substituting the series into \eqref{ysys}, we get
\begin{gather}\label{xikzetak2}
\begin{split}
-(\lambda_n \xi_k+ \nu_n \zeta_k)=&\,\mathfrak F_k(\xi_1,\zeta_1,\dots,\xi_{k-1},\zeta_{k-1}), \\
 \frac{d\zeta_k}{d\tau}-\left(a_2^0+ \frac{k}{q}\right)\zeta_k=&\,\tilde {\mathfrak G}_k(\xi_1,\zeta_1,\dots,\xi_{k-1},\zeta_{k-1}), \quad k\geq 1,
\end{split}
\end{gather}
where $\tilde {\mathfrak  G}_k=\mathfrak G_k-\lambda_n^{-1}\eta_m \mathfrak F_k$ and $a_2^0=\lambda_n^{-1}\mathcal D_{n,m}$. Since $a_2^0<0$, it follows that system \eqref{xikzetak2} is solvable and
\begin{gather*}
\zeta_k(\tau)\equiv 
	\begin{cases}
	 \displaystyle -\frac{q \tilde{\mathfrak G}_k}{ q \alpha_2^0+k}, &\text{if} \quad  k<q|\alpha_2^0|,\\
	\displaystyle \int\limits_{\log t_0}^{\tau} \tilde{\mathfrak G}_k\,d\theta, &\text{if} \quad  k=q|\alpha_2^0|,\\
	\displaystyle \int\limits_{\infty}^{\tau} e^{-(\alpha_2^0+k/q)(\theta-\tau)}\tilde{\mathfrak G}_k\,d\theta, &\text{if} \quad  k>q|\alpha_2^0|.
	\end{cases}
\end{gather*}
Define the functions
\begin{gather*}
\hat \Xi_M(t)\equiv \sum_{k=1}^M t^{-\frac{k}{q}}\xi_k(\log t), \quad 
\hat Z_M(t)\equiv \sum_{k=1}^M t^{-\frac{k}{q}}\zeta_k(\log t), \quad M\in\mathbb Z_+.
\end{gather*}
Then, the following estimates holds:
\begin{align*}
&\hat{\mathfrak R}_1(t)\equiv \frac{d\hat\Xi_M(t)}{dt}-\mathcal F(\hat\Xi_M(t),\hat Z_M(t),t)=\mathcal O\left(t^{-\frac{n+M+1}{q}}\right), \\
&\hat{\mathfrak R}_2(t)\equiv\frac{d\hat Z_M(t)}{dt}-\mathcal G(\hat\Xi_M(t),\hat Z_M(t),t)=\mathcal O\left(t^{-1-\frac{M+1}{q}}\right), \quad t\to\infty.
\end{align*}
Substituting $y_1(t)=\hat \Xi_M(t)+t^{-\epsilon} z_1(t)$, $y_2(t)=\hat Z_M(t)+t^{-\epsilon} z_2(t)$ with some $\epsilon\in(0,1/q)$ into \eqref{ysys2} we obtain
\begin{gather}\label{z1z2sys2}
\frac{dz_1}{dt}=\hat{\mathcal F}(z_1,z_2,t)-t^{ \epsilon} \hat {\mathfrak R}_1(t), \quad 
\frac{dz_2}{dt}=\hat{\mathcal G}(z_1,z_2,t)-t^{ \epsilon}  \hat {\mathfrak R}_2(t),
\end{gather} 
where
\begin{align*}
	\hat {\mathcal F}(z_1,z_2,t)& \equiv  t^{\epsilon} \left(\mathcal F(\hat \Xi_M(t)+t^{-\epsilon} z_1,\hat Z_M(t)+t^{-\epsilon} z_2,t)-\mathcal F(\hat \Xi_M(t),\hat Z_M(t),t)\right)+\epsilon t^{-1} z_1, \\
	 & = t^{-\frac{n}{q}}\left(\beta_{n,q,1} z_1+\nu_n z_2+\mathcal O(t^{-\epsilon})\mathcal O(|{\bf z}|)\right),\\
	\hat {\mathcal G}(z_1,z_2,t)& \equiv    t^{\epsilon} \left( \mathcal G(\hat \Xi_M(t)+t^{-\epsilon}z_1,\hat Z_M(t)+t^{-\epsilon}z_2,t)-\mathcal G(\hat \Xi_M(t),\hat Z_M(t),t)\right)+\epsilon t^{-1} z_2\\
	& =t^{-1}\left(\eta_q z_1+(\omega_q+\epsilon) z_2+\mathcal O(t^{-\epsilon})\mathcal O(|{\bf z}|)\right) 
\end{align*}
as $|{\bf z}|\to 0$ and $t\to\infty$. 
We choose $\epsilon$ small enough such that $0<\epsilon<\min\{|\beta_{n,q,2}|,1/q\}$ and consider $\hat V(z_1,z_2,t)\equiv L(z_1,z_2,t)$ as a Lyapunov function candidate for system \eqref{z1z2sys}, where
$L(z_1,z_2,t)$ is defined by \eqref{LFunc} with 
\begin{gather*}
C_1=
\begin{cases}
\displaystyle \beta_{n,q,1} \left(\omega_q+\frac{q-n}{2q}+\epsilon\right), & \text{if} \quad \nu_n=0,\\
\displaystyle \frac{\beta_{n,q,1}(\beta_{n,q,2}+\epsilon)}{2\nu_n^2}, & \text{if} \quad \nu_n\neq 0,
\end{cases}
\quad  
C_2=-\frac{2C_1\nu_n+2\eta_q}{\beta_{n,q,1}}.
\end{gather*}
It can easily be checked that $C_1>0$.
The derivative of $\hat V(z_1,z_2,t)$ along the trajectories of system \eqref{z1z2sys2} satisfies
\begin{gather*}
\frac{d\hat V}{dt}\Big|_{\eqref{z1z2sys2}}= 2C_1 \beta_{n,q,1} t^{-\frac{n}{q}} \left\{z_1^2+(\nu_n^2+\delta_{\nu_n,0})\beta_{n,q,1}^{-2}z_2^2+\mathcal O(t^{-\epsilon})\mathcal O(|{\bf z}|^2)\right\}+\mathcal O(t^{-\frac{n+M}{q}})\mathcal O(|{\bf z}|)
\end{gather*}
as $|{\bf z}|\to 0$ and $t\to\infty$. Hence, there exist $\Delta_2>0$ and $t_2\geq t_0$ such that
\begin{gather*}
\hat B_- |{\bf z}|^2\leq \hat V(z_1,z_2,t)\leq \hat B_+ t^{1-\frac{n}{q}}|{\bf z}|^2,\quad
\frac{d\hat V}{dt}\Big|_{\eqref{z1z2sys2}}\leq t^{-\frac{n}{q}}\left(-\hat\gamma |{\bf z}|^2+ t^{-\frac{M}{q}} \hat B_\ast |{\bf z}|\right)
\end{gather*}
for all $|{\bf z}|\leq \Delta_2$ and $t\geq t_2$ with some constants $\hat B_\pm>0$, $\hat B_\ast>0$ and $\hat \gamma= C_1|\beta_{n,q,1}|\min\{1,(\nu_n^2+\delta_{\nu_n,0})\beta_{n,q,1}^{-2}\}>0$. 
Let $M> (q-n)/2$. Then, for all $\varepsilon\in (0,\Delta_2)$ there exist
\begin{gather*}
\delta_\varepsilon=\frac{2\hat B_\ast }{\hat \gamma}t_\varepsilon^{-\frac{2M-(q-n)}{2q}}, \quad
t_\varepsilon=\max\left\{1, t_2, \left(\frac{4\hat B_\ast}{\hat \gamma \varepsilon}\sqrt{\frac{\hat B_+}{\hat B_-}}\right)^{\frac{ q}{M}}\right\}
\end{gather*}
such that
\begin{gather*}
\frac{d\hat V}{dt}\Big|_{\eqref{z1z2sys2}}\leq t^{-\frac{n}{q}}|{\bf z}|^2\left(-\hat \gamma+t_\varepsilon^{-\frac{2M-(q-n)}{2q}}\frac{\hat B_\ast}{\delta_\varepsilon}\right)\leq -t^{-\frac{n}{q}}\frac{\hat \gamma |{\bf z}|^2}{2} <0
\end{gather*}
for all $\delta_\varepsilon  t^{-\frac{q-n}{2q}}\leq |{\bf z}|\leq \varepsilon$ and $t\geq t_\varepsilon$. Moreover, the function $\hat V(z_1,z_2,t)$ satisfies the inequality \eqref{Vineq} with $m=q$. Hence, any solution $(z_1(t),z_2(t))$ of system \eqref{z1z2sys2} with initial data $|{\bf z}(t_s)|\leq \delta_\varepsilon t_s^{-\frac{q-n}{2q}}$, $t_s\geq t_\varepsilon$ cannot exit from the domain $|{\bf z}|\leq\varepsilon$ as $t> t_s$. Thus, for all $M>(q-n)/2$ there exists a solution of system \eqref{ysys2} such that $y_1(t)=\hat \Xi_M(t)+\mathcal O(t^{-\epsilon})$ and $y_2(t)=\hat Z_M(t)+\mathcal O(t^{-\epsilon})$ as $t\to\infty$.
\end{proof}

\section{Analysis of the full system}\label{sec5}

\begin{proof}[Proof of Theorem~\ref{Th2}]
Consider the case $n\neq m$. For definiteness, assume that $n<m$.

{\bf 1}. Let $\beta_{n,m,1}<0$ and $\beta_{n,m,2}<0$.
Substituting $r(t)=\rho_\ast(t)+y_1(t)$, $\phi(t)=\phi_\ast(t)+y_2(t)$ into \eqref{rpsi}, we obtain
\begin{gather}\label{ysys3}
 \frac{dy_1}{dt}= \widehat{\mathcal F}_N(y_1,y_2,t)+\widetilde{\mathcal F}_N(y_1,y_2,t), \quad 
\frac{dy_2}{dt}= \widehat{\mathcal G}_N(y_1,y_2,t)+\widetilde{\mathcal G}_N(y_1,y_2,t), 
\end{gather}
where
\begin{align*}
\widehat{\mathcal F}_N(y_1,y_2,t)=&\, \widehat\Lambda_N(\rho_\ast(t)+y_1,\phi_\ast(t)+y_2,t)-\widehat\Lambda_N(\rho_\ast(t),\phi_\ast(t),t), \\ 
\widehat{\mathcal G}_N(y_1,y_2,t)=&\, \widehat\Omega_N(\rho_\ast(t)+y_1,\phi_\ast(t)+y_2,t)-\widehat\Omega_N(\rho_\ast(t),\phi_\ast(t),t), \\
\widetilde{\mathcal F}_N(y_1,y_2,t)=&\, \widetilde\Lambda_N(\rho_\ast(t)+y_1,\phi_\ast(t)+y_2,S(t),t), \\ 
\widetilde{\mathcal G}_N(y_1,y_2,t)=&\, \widetilde\Omega_N(\rho_\ast(t)+y_1,\phi_\ast(t)+y_2,S(t),t).
\end{align*}
It can easily be checked that
\begin{align*}
	 & \widehat{\mathcal F}_N(y_1,y_2,t)=t^{-\frac{n}{q}}\left(\lambda_n y_1+\nu_n y_2+\mathcal O(|{\bf y}|^2)+\mathcal O(t^{-\frac{1}{q}})\mathcal O(|{\bf y}|)\right), \\ 
	& \widehat{\mathcal G}_N(y_1,y_2,t)=t^{-\frac{m}{q}}\left(\eta_m y_1+\omega_m y_2+\mathcal O(|{\bf y}|^2)+\mathcal O(t^{-\frac{1}{q}})\mathcal O(|{\bf y}|)\right),\\
	&  \widetilde{\mathcal F}_N(y_1,y_2,t)=\mathcal O(t^{-\frac{N+1}{q}}),\\
	&  \widetilde{\mathcal G}_N(y_1,y_2,t)=\mathcal O(t^{-\frac{N+1}{q}})
\end{align*}
as $|{\bf y}|=\sqrt{y_1^2+y_2^2}\to 0$ and $t\to\infty$. Consider the Lyapunov function candidate for system \eqref{ysys3} in the form $V(y_1,y_2,t)\equiv L(y_1,y_2,t)$, where
$L(y_1,y_2,t)$ is defined by \eqref{LFunc} and \eqref{C1C2}. The derivative of $V(y_1,y_2,t)$ along the trajectories of system \eqref{ysys3} satisfies
\begin{align*}
\frac{dV}{dt}\Big|_{\eqref{ysys3}}= &  2C_1 \beta_{n,m,1} t^{-\frac{n}{q}} \left\{y_1^2+(\nu_n^2+\delta_{\nu_n,0})\beta_{n,m,1}^{-2}y_2^2 +\mathcal O(|{\bf y}|^3)+\mathcal O(t^{-\frac{1}{q}})\mathcal O(|{\bf y}|^2)\right\}\\
& +\mathcal O(t^{-\frac{N+1+n-m}{q}})\mathcal O(|{\bf y}|)
\end{align*}
as $|{\bf y}|\to 0$ and $t\to\infty$. Hence, there exist $\Delta_1>0$ and $t_1\geq t_0$ such that 
\begin{gather*} 
B_- |{\bf y}|^2\leq V(y_1,y_2,t)\leq B_+ t^{\frac{m-n}{q}}|{\bf y}|^2,\quad
\frac{dV}{dt}\Big|_{\eqref{ysys3}}\leq t^{-\frac{n}{q}}\left(-\gamma |{\bf y}|^2+ t^{-\frac{N+1-m}{q}} B_\ast |{\bf y}|\right)
\end{gather*}
 for all $|{\bf y}|\leq \Delta_1$ and $t\geq t_1$ with some constants $B_\pm>0$, $B_\ast>0$ and $\gamma=C_1|\beta_{n,m,1}|\min\{1,(\nu_n^2+\delta_{\nu_n,0})\beta_{n,m,1}^{-2}\}>0$. Taking $N\geq (3m-n)/2$ ensures that for all $\varepsilon\in (0,\Delta_1)$ there exist
\begin{gather*}
\delta_\varepsilon=\frac{2B_\ast }{\gamma}t_\varepsilon^{-\frac{1}{q}}, \quad
t_\varepsilon=\max\left\{1, t_1, \left(\frac{4B_\ast}{\gamma \varepsilon}\sqrt{\frac{B_+}{B_-}}\right)^{q}\right\}
\end{gather*}
such that
\begin{gather*}
\frac{dV}{dt}\Big|_{\eqref{ysys3}}\leq t^{-\frac{n}{q}}|{\bf y}|^2\left(-\gamma+t_\varepsilon^{-\frac{1}{q}}\frac{B_\ast}{\delta_\varepsilon}\right)\leq -t^{-\frac{n}{q}}\frac{\gamma |{\bf y}|^2}{2} <0
\end{gather*}
for all $\delta_\varepsilon  t^{-\frac{m-n}{2q}}\leq |{\bf y}|\leq \varepsilon$ and $t\geq t_\varepsilon$. Combining this with the inequalities \eqref{Vineq} we see that any solution $(y_1(t),y_2(t))$ of system \eqref{ysys3} with initial data $|{\bf y}(t_s)|\leq \delta_\varepsilon t_s^{-(m-n)/(2q)}$, $t_s\geq t_\varepsilon$ cannot exit from the domain $|{\bf y}|\leq\varepsilon$ as $t> t_s$.  Returning to the variables $(r,\phi)$ and taking into account \eqref{ch1} and \eqref{ch2}, we obtain \eqref{ineqrvarphi}.

{\bf 2}. Now, let $\beta_{n,m,1}>0$, $\beta_{n,m,2}>0$ and $\max\{n,m\}<q$.  Substituting $r(t)=\varrho_\ast+y_1(t)$, $\phi(t)=\varphi_\ast+y_2(t)$ into \eqref{rpsi}, we obtain system \eqref{ysys3} with 
\begin{align*}
\widehat{\mathcal F}_N(y_1,y_2,t)\equiv  &\, \widehat\Lambda_N(\varrho_\ast+y_1,\varphi_\ast+y_2,t), \\ 
\widehat{\mathcal G}_N(y_1,y_2,t)\equiv &\, \widehat\Omega_N(\varrho_\ast+y_1,\varphi_\ast+y_2,t), \\
\widetilde{\mathcal F}_N(y_1,y_2,t)\equiv &\, \widetilde\Lambda_N(\varrho_\ast+y_1,\varphi_\ast+y_2,S(t),t), \\ 
\widetilde{\mathcal G}_N(y_1,y_2,t)\equiv &\, \widetilde\Omega_N(\varrho_\ast+y_1,\varphi_\ast+y_2,S(t),t).
\end{align*}
It follows from \eqref{R1R2} that
\begin{align*}
	 & \widehat{\mathcal F}_N(y_1,y_2,t)=t^{-\frac{n}{q}}\left(\lambda_n y_1+\nu_n y_2+\mathcal O(|{\bf y}|^2)+\mathcal O(t^{-\frac{1}{q}})\right), \\ 
	& \widehat{\mathcal G}_N(y_1,y_2,t)=t^{-\frac{m}{q}}\left(\eta_m y_1+\omega_m y_2+\mathcal O(|{\bf y}|^2)+\mathcal O(t^{-\frac{1}{q}})\right),\\
	&  \widetilde{\mathcal F}_N(y_1,y_2,t)=\mathcal O(t^{-\frac{N+1}{q}}),\\
	&  \widetilde{\mathcal G}_N(y_1,y_2,t)=\mathcal O(t^{-\frac{N+1}{q}})
\end{align*}
as $|{\bf y}|\to 0$ and $t\to\infty$. The Lyapunov function candidate is considered in the form $V(y_1,y_2,t)\equiv L(y_1,y_2,t)$, where
$L(y_1,y_2,t)$ is defined by \eqref{LFunc} and \eqref{C1C2}. 
The derivative of $V(y_1,y_2,t)$ is given by 
\begin{align*}
\frac{dV}{dt}\Big|_{\eqref{ysys3}}= &2C_1 \beta_{n,m,1} t^{-\frac{n}{q}} \left\{y_1^2+(\nu_n^2+\delta_{\nu_n,0})\beta_{n,m,1}^{-2}y_2^2 +\mathcal O(|{\bf y}|^3)+\mathcal O(t^{-\frac{1}{q}})\mathcal O(|{\bf y}|)\right\} \\
& +\mathcal O(t^{-\frac{N+1+n-m}{q}})\mathcal O(|{\bf y}|)
\end{align*}
as $|{\bf y}|\to 0$ and $t\to\infty$. Hence, there exist $N_0\geq m$, $\varepsilon>0$ and $t_1\geq t_0$ such that  
\begin{gather*} 
B_- |{\bf y}|^2\leq V(y_1,y_2,t)\leq B_+ t^{\frac{m-n}{q}}|{\bf y}|^2,\quad
\frac{dV}{dt}\Big|_{\eqref{ysys3}}\geq t^{-\frac{n}{q}}\left(\gamma |{\bf y}|^2- t^{-\frac{1}{q}} B_\ast |{\bf y}|\right)
\end{gather*}
 for all $N\geq N_0$, $|{\bf y}|\leq \varepsilon$ and $t\geq t_1$ with some constants $B_\pm>0$, $B_\ast>0$ and $\gamma=C_1|\beta_{n,m,1}|\min\{1,(\nu_n^2+\delta_{\nu_n,0})\beta_{n,m,1}^{-2}\}>0$. It follows that for all $\delta\in (0,\varepsilon)$ there exists  
$t_\ast=\max\{1, t_1,({2C_\ast}/(\gamma \delta))^{q}\}$
such that
\begin{gather*}
\frac{dV}{dt}\Big|_{\eqref{ysys3}}\geq t^{-\frac{n}{q}}|{\bf y}|^2\left(\gamma-t_\ast^{-\frac{1}{q}}\frac{B_\ast}{\delta}\right)\geq t^{-\frac{n}{q}} \frac{\gamma |{\bf y}|^2}{2}\geq t^{-\frac{m}{q}} \hat \gamma V
\end{gather*}
for all $\delta  \leq |{\bf y}|\leq \varepsilon$ and $t\geq t_\ast$ with $\hat \gamma=\gamma / (2 B_+)$.   Integrating the last inequality as $t\geq t_s\geq t_\ast$ and taking $|{\bf y}(t_s)|=\delta$, we obtain
\begin{gather*}
|{\bf y}(t)|^2\geq \frac{\delta^2 B_-}{B_+} t^{-\frac{m-n}{q}}\exp\left\{\frac{q\hat \gamma}{q-m}\left(t^{1-\frac{m}{q}}-t_s^{1-\frac{m}{q}}\right)\right\}.
\end{gather*}
Therefore, there exists $t_e>t_s$ such that $|{\bf y}(t_e)|\geq \varepsilon$. Returning to the variables $(\varrho,\varphi)$, we obtain \eqref{ineqrvarphi2}.

Consider now the case $n=m$. We substitute ${\bf y}={\bf T}_0 {\bf z}$ into \eqref{ysys3} with the matrix ${\bf T}_0$ defined by \eqref{T0}. Then, taking the Lyapunov function candidate in the form \eqref{L0} and repeating the arguments used above, we obtain \eqref{ineqrvarphi} or \eqref{ineqrvarphi2}.
\end{proof}

\begin{proof}[Proof of Theorem~\ref{Th3}]
Substituting $r(t)=\varrho_\ast+y(t)$ into \eqref{rpsi}, we obtain
\begin{gather}\label{ypsi}
\frac{dy}{dt}=\widehat {\mathcal F}_N(y,\psi,t)+\widetilde {\mathcal F}_N(y,\psi,t), \quad 
\frac{d\psi}{dt}=\widehat {\Omega}_N(\varrho_\ast+y,\psi,t)+\widetilde {\Omega}_N(\varrho_\ast+y,\psi,S(t),t),
\end{gather}
where $\widehat {\mathcal F}_N(y,\psi,t)\equiv \widehat\Lambda_N(\varrho_\ast+y,\psi,t)$, $\widetilde {\mathcal F}_N(y,\psi,t)\equiv \widetilde\Lambda_N(\varrho_\ast+y,\psi,S(t),t)$. 
Note that 
\begin{gather*}
	  \widehat{\mathcal F}_N(y,\psi,t)=t^{-\frac{n}{q}}\left(\ell_n(\psi) y+\mathcal O(y^2)+\mathcal O(t^{-\frac{1}{q}})\right), \quad
		\widetilde{\mathcal F}_N(y,\psi,t)=\mathcal O(t^{-\frac{N+1}{q}})
\end{gather*}
as $y\to 0$ and $t\to\infty$ uniformly for all $\psi\in\mathbb R$. Since $\ell_n(\psi)\neq 0$ for all $\psi\in\mathbb R$, we see that there exists $\ell_\ast>0$ such that $|\ell_n(\psi)|\geq \ell_\ast$ for all $\psi\in\mathbb R$. It follows that there exist $\Delta_1>0$ and $t_1\geq t_0$ such that 
\begin{align*}
\frac{d|y|}{dt}\leq & \, - t^{-\frac{n}{q}}\left(\frac{\ell_\ast |y|}{2}- t^{-\frac{1}{q}}B_\ast\right), & \left|\frac{d\psi}{dt}\right|&\geq t^{-\frac{m}{q}} P_\ast & \text{if} \quad & \ell_n(\psi)<0, \\
\frac{d|y|}{dt}\geq & \, t^{-\frac{n}{q}}\left(\frac{\ell_\ast |y|}{2}- t^{-\frac{1}{q}}B_\ast\right), & \left|\frac{d\psi}{dt}\right|&\geq t^{-\frac{m}{q}} P_\ast & \text{if} \quad & \ell_n(\psi)>0
\end{align*}
for all $|y|\leq \Delta_1$, $\psi\in\mathbb R$ and $t\geq t_1$ with some constants $B_\ast>0$ and $P_\ast>0$.

If $\ell_n(\psi)<0$, then for all $\varepsilon\in (0,\Delta_1)$ there exist
\begin{gather*}
\delta_\varepsilon=\frac{4B_\ast t_\varepsilon^{-\frac{1}{q}}}{\ell_\ast}, \quad t_\varepsilon=\max\left\{t_1,\left(\frac{8 B_\ast}{\ell_\ast \varepsilon}\right)^q\right\}
\end{gather*}
such that 
\begin{gather*}
\frac{d|y|}{dt}\leq -t^{-\frac n q}|y|\left(\frac{\ell_\ast}{2}-t_\varepsilon^{-\frac 1q}\frac{B_\ast}{\delta_\varepsilon}\right)\leq -t^{-\frac n q}\frac{\ell_\ast|y|}{4}
\end{gather*}
for all $\delta_\varepsilon\leq |y|\leq \varepsilon$, $\psi\in\mathbb R$ and $t\geq t_\varepsilon$. Therefore, any solution $(y(t),\psi(t))$ of system \eqref{ypsi} with initial data $|y(t_s)|\leq \delta_\varepsilon$, $\psi(t_s)\in\mathbb R$, $t_s\geq t_\varepsilon$ satisfies the inequality $|y(t)|<\varepsilon$ for all $t\geq t_s$. From the second equation in \eqref{ypsi} it follows that $\psi(t)\to \infty$ as $t\to\infty$.

If $\ell_n(\psi)>0$, then for all $\delta\in (0,\Delta_1)$ there exists $t_\ast=\max\{t_1,(4B_\ast/(\ell_\ast \delta))^q\}$ such that $d|y|/dt\geq  t^{-n/q}|y|(\ell_\ast/2-t_\ast^{-1/q}B_\ast/\delta)\geq t^{-n/q}\ell_\ast|y|/4$ for all $\delta\leq |y|\leq \Delta_1$, $\psi\in\mathbb R$ and $t\geq t_\ast$. Integrating the last inequality as $t\geq t_s$ with $|y(t_\ast)|=\delta$ and $t_s\geq t_\ast$, we obtain 
\begin{gather*}
|y(t)|\geq \delta \exp\left(\frac{\ell_\ast}{4}\int\limits_{t_s}^t \tau^{-\frac{n}{q}}\, d\tau\right).
\end{gather*}
Therefore, there exists $t_e>t_s$ such that $|y(t_e)|=\Delta_1$.
\end{proof}

\section{Examples}\label{sec6}

\subsection{Example 1.} Consider again system \eqref{Ex0}. It was shown in Section~\ref{sec1} that this system corresponds to \eqref{PS} with $q=\omega=1$ and the functions $f(\varrho,\varphi,S,t)$, $g(\varrho,\varphi,S,t)$ defined by \eqref{fgex0}. Let $s_0=1$, then there exist $\kappa=\varkappa=1$ such that the resonance condition \eqref{rc} holds. It can easily be checked that the change of variables described in Theorem~\ref{Th1} with $N=1$ transforms the system into
\begin{gather*}
\begin{split}
& \frac{dr}{dt}=t^{-1}\Lambda_1(r,\psi)+\widetilde \Lambda_1(r,\psi,S(t),t), \\
& \frac{d\psi}{dt}=t^{-1}\Omega_1(r,\psi)+\widetilde \Omega_1(r,\psi,S(t),t),
\end{split}
\end{gather*}
where
\begin{gather*}
\Lambda_1(r,\psi)\equiv -\frac{(4a+3c r^2)\cos\psi}{8}, \quad 
\Omega_1(r,\psi)\equiv -s_1+ \frac{(4a+c r^2)\sin\psi}{8r},
\end{gather*}
and $\widetilde \Lambda_1(r,\psi,S,t)=\mathcal O(t^{-2})$, $\widetilde \Omega_1(r,\psi,S,t)=\mathcal O(t^{-2})$ as $t\to\infty$ uniformly for all $r\in (0,\mathcal R]$, $(\psi,S)\in\mathbb R^2$ with any fixed $\mathcal R>0$. We see that assumption \eqref{asnm} holds with $n=m=1$.  

Let $a c<0$ and $12 s_1^2<|a c|$. Then, assumption \eqref{as1} holds with
\begin{gather*}
\varrho_\ast=\sqrt{-\frac{4a}{3c}}, \quad \varphi_\ast= (-1)^k \theta_\ast +\pi k, \quad k\in\mathbb Z, \quad \theta_\ast=\arcsin\frac{3s_1 \varrho_\ast}{a}.
\end{gather*}
It can easily be checked that 
\begin{gather*}
\lambda_n=-\frac{3c\varrho_\ast\cos\varphi_\ast}{4}, \quad 
\nu_n=0,\quad 
\eta_m=\frac{c\sin\varphi_\ast}{2}, \quad 
\omega_m=-\frac{c\varrho_\ast \cos\varphi_\ast}{4},
\end{gather*}
$\alpha_1^0=\beta_{n,m,1}=\lambda_n$ and $\alpha_2^0=\beta_{n,m,2}=\omega_m$.
It follows from Lemma~\ref{Lem1} and Theorem~\ref{Th2} that for all $k\in\mathbb Z$ the equilibrium $(\varrho_\ast,(-1)^k\theta_\ast+\pi k)$ in the limiting system and the corresponding regime in system \eqref{Ex0} are unstable if $(-1)^k c<0$. It also follows that if $(-1)^k c>0$, then the persistent phase-locking occurs and there are stable regimes for solutions of system \eqref{Ex0} with $\varrho(t)\approx \varrho_\ast$ and $\varphi(t)\approx  S(t)+(-1)^k\theta_\ast+\pi k$ as $t\to\infty$ (see~Fig.~\ref{FigEx1} and Fig.~\ref{FigEx0}, c).

\begin{figure}
\centering
{
   \includegraphics[width=0.4\linewidth]{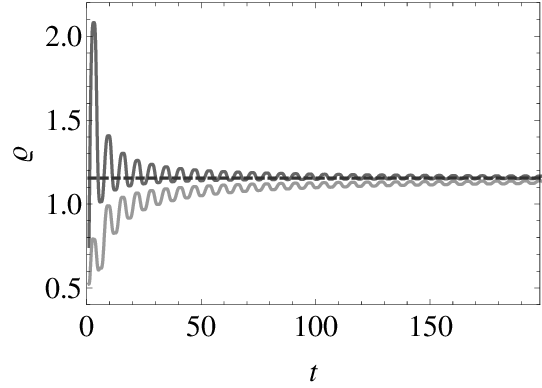}
}
\hspace{1ex}
{
   	\includegraphics[width=0.4\linewidth]{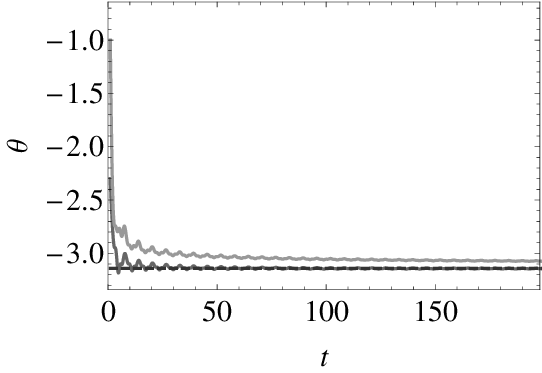}
}
\caption{\small The evolution of $\varrho(t)\equiv \sqrt{x_1^2(t)+x_2^2(t)}$ and $\theta(t)\equiv \varphi(t)-S(t)$, $\tan \varphi(t)=-x_2(t)/x_1(t)$ for solutions to system \eqref{Ex0} with $s_0=1$, $s_1=0$, $a=1$, $b=2$, $c=-1$ and different values of initial data. The dashed curves correspond to $\varrho(t)\equiv \varrho_\ast$ and $\theta(t)\equiv -\pi$, where $\varrho_\ast=2/\sqrt{3}$.} \label{FigEx1}
\end{figure}

\subsection{Example 2.} Consider 
\begin{gather}\label{Ex2}
\frac{dx_1}{dt}=x_2, \quad \frac{dx_2}{dt}=-x_1+t^{-\frac{1}{2}}\mathcal Z(x_2,S(t)),
\end{gather}
where
\begin{gather*}
\mathcal Z(x,S)\equiv (b(S)+c(S) x^2)x, \quad 
b(S)\equiv b_0+b_1\sin S, \quad 
c(S)\equiv c_0+c_1\sin S,
\end{gather*}
$S(t)\equiv 2 t+2 s_1 t^{1/2}$ with $b_i,c_i,s_i\in\mathbb R$. It can easily be checked that system \eqref{Ex2} in the variables $\varrho=\sqrt{x_1^2+x_2^2}$, $\varphi=-\arctan(x_2/x_1)$ takes the form \eqref{PS} with $q=2$, $\omega=1$, 
\begin{gather*}
f(\varrho,\varphi,S,t)\equiv -t^{-\frac{1}{2}}\mathcal Z(-\varrho\sin\varphi,S)\sin\varphi, \quad
g(\varrho,\varphi,S,t)\equiv -t^{-\frac{1}{2}}\varrho^{-1}\mathcal Z(-\varrho\sin\varphi,S)\cos\varphi.
\end{gather*} 
Note that the resonance condition \eqref{rc} is satisfied with $\kappa=1$, $\varkappa=2$ such that. Moreover, the transformation described in Theorem~\ref{Th1} with $N=1$ reduces the system to
\begin{align*}
& \frac{dr}{dt}=t^{-\frac12}\Lambda_1(r,\psi)+\widetilde \Lambda_1(r,\psi,S(t),t), \\
& \frac{d\psi}{dt}=t^{-\frac12}\Omega_1(r,\psi)+\widetilde \Omega_1(r,\psi,S(t),t),
\end{align*}
where
\begin{gather*}
\Lambda_1(r,\psi)\equiv \frac{r}{8} \left(4b_0+3 c_0 r^2 + 2(b_1 + c_1 r^2) \sin 2 \psi\right), \quad 
\Omega_1(r,\psi)\equiv \frac{1}{8} \left(-4 s_1 + (2 b_1 + c_1 r^2) \cos 2 \psi\right),
\end{gather*}
and $\widetilde \Lambda_1(r,\psi,S,t)=\mathcal O(t^{-1})$, $\widetilde \Omega_1(r,\psi,S,t)=\mathcal O(t^{-1})$ as $t\to\infty$ uniformly for all $r\in (0,\mathcal R]$, $(\psi,S)\in\mathbb R^2$ with any $\mathcal R>0$. It follows that assumption \eqref{asnm} holds with $n=m=1$.  

Let $2b_0/b_1=3c_0/(2c_1)=\delta>1$. If $b_1 c_1<0$ and $|b_1|<4|s_1|$, then there is $\varrho_\ast=\sqrt{-b_1/c_1}$ such that the condition \eqref{as2} holds with $\ell_n(\psi)=-b_1 (\delta+\sin 2\psi)/2$ and $\Omega_m(\varrho_\ast,\psi)\equiv (b_1 \cos 2\psi-4s_1)/8\neq 0$ for all $\psi\in\mathbb R$.
It follows from Theorem~\ref{Th3} that if $b_1>0$, then there is a stable regime for solutions of system \eqref{PS} such that $\varrho(t)\approx \varrho_\ast$ and $\varphi-S(t)/2\to \infty$ as $t\to\infty$. This regime is unstable if $b_1<0$ (see~Fig.~\ref{FigEx21}). 
\begin{figure}
\centering
{
   \includegraphics[width=0.4\linewidth]{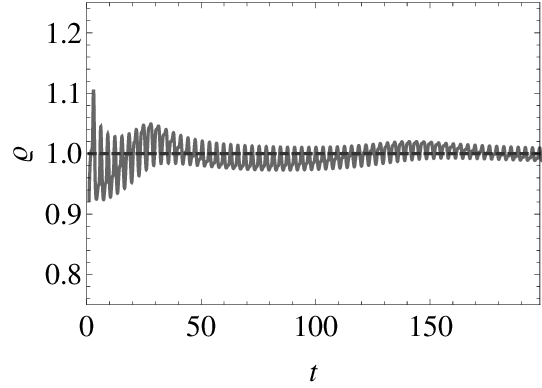}
}
\hspace{1ex}
{
   	\includegraphics[width=0.4\linewidth]{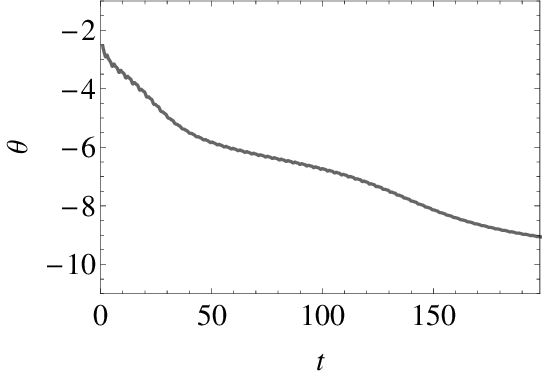}
}
\caption{\small The evolution of $\varrho(t)\equiv \sqrt{x_1^2(t)+x_2^2(t)}$ and $\theta(t)\equiv \varphi(t)-S(t)/2$, $\tan \varphi(t)=-x_2(t)/x_1(t)$ for solutions to system \eqref{Ex2} with $s_1=1/2$, $b_0=3/2$, $b_1=1$, $c_0=-2$, $c_1=-1$. The dashed curve corresponds to $\varrho(t)\equiv \varrho_\ast$, where $\varrho_\ast=1$.} \label{FigEx21}
\end{figure}

If $b_1 c_1<0$ and $|b_1|>4|s_1|$, then assumption \eqref{as1} holds with 
\begin{gather*}
\varrho_\ast=\sqrt{-\frac{b_1}{c_1}}, \quad \varphi_\ast= \pm \theta_\ast+\pi k, \quad k\in\mathbb Z, \quad 
\theta_\ast=\frac{1}{2}\arccos \frac{4s_1}{b_1}.
\end{gather*}
In this case, we have
\begin{gather*}
\lambda_n=-\frac{b_1 (\delta+\sin 2\varphi_\ast)}{2}, \quad 
\nu_n=0,\quad 
\eta_m=\frac{c_1 \varrho_\ast \cos 2\varphi_\ast}{2}, \quad 
\omega_m=-\frac{b_1 \sin 2\varphi_\ast}{4},
\end{gather*}
$a_1^0=d_{n,m,1}=\lambda_n$ and $a_2^0=d_{n,m,2}=\omega_m$.
From Lemma~\ref{Lem1} and Theorem~\ref{Th2} it follows that if $b_1<0$, then for all $k\in\mathbb Z$ the equilibrium $(\varrho_\ast,\theta_\ast+\pi k)$ in the limiting system and the corresponding regime in system \eqref{Ex2} are both unstable. If $b_1>0$, then the persistent phase-locking occurs and there are stable regimes for solutions of system \eqref{Ex2} with $\varrho(t)\approx \varrho_\ast$ and $\varphi(t)\approx  S(t)/2+\theta_\ast+\pi k$ as $t\to\infty$ (see~Fig.~\ref{FigEx22}). 

\begin{figure}
\centering
{
   \includegraphics[width=0.4\linewidth]{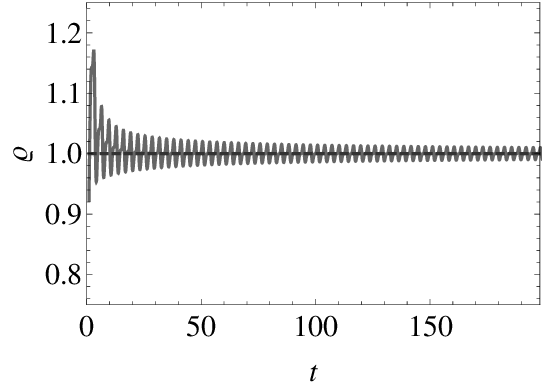}
}
\hspace{1ex}
{
   	\includegraphics[width=0.4\linewidth]{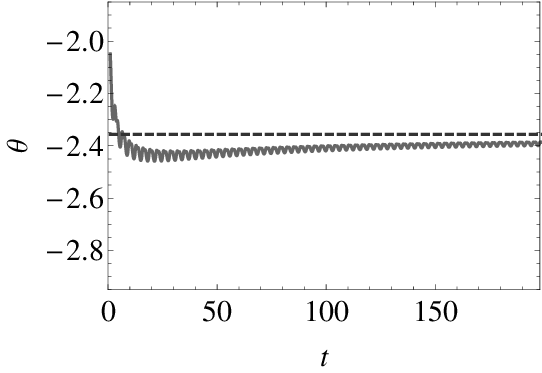}
}
\caption{\small The evolution of $\varrho(t)\equiv \sqrt{x_1^2(t)+x_2^2(t)}$ and $\theta(t)\equiv \varphi(t)-S(t)/2$, $\tan \varphi(t)=-x_2(t)/x_1(t)$ for solutions to system \eqref{Ex2} with $s_1=0$, $b_0=3/2$, $b_1=1$, $c_0=-2$, $c_1=-1$. The dashed curves correspond to $\varrho(t)\equiv \varrho_\ast$ and $\theta(t)\equiv \theta_\ast-\pi$, where $\varrho_\ast=1$ and $\theta_\ast=\pi/4$.} \label{FigEx22}
\end{figure}

\subsection{Example 3.} 
Finally, consider the system
\begin{gather}\label{Ex3}
\frac{dx_1}{dt}=x_2, \quad \frac{dx_2}{dt}=-x_1+t^{-\frac{1}{2}}\mathcal Z_1(x_2)+t^{-1}\mathcal Z_2(x_2,S(t)),
\end{gather}
where
\begin{gather*}
\mathcal Z_1(x)\equiv (b_0+c_0 x^2)x, \quad 
\mathcal Z_2(x,S)\equiv b_1 x \sin S, \quad 
S(t)\equiv 2 t+ s_2 \log t
\end{gather*}
with $b_0,b_1,c_0,s_2\in\mathbb R$.
Note that this system  in the variables $\varrho=\sqrt{x_1^2+x_2^2}$, $\varphi=-\arctan(x_2/x_1)$  takes the form \eqref{PS} with $q=2$, $\omega=1$,
\begin{align*}
f(\varrho,\varphi,S,t)&\equiv -\left(t^{-\frac{1}{2}}\mathcal Z_1(-\varrho\sin\varphi)+t^{-1}\mathcal Z_2(-\varrho\sin\varphi,S)\right)\sin\varphi, \\
g(\varrho,\varphi,S,t)&\equiv -\varrho^{-1}\left(t^{-\frac{1}{2}}\mathcal Z_1(-\varrho\sin\varphi)+t^{-1}\mathcal Z_2(-\varrho\sin\varphi,S)\right)\cos\varphi.
\end{align*} 
We see that the resonance condition \eqref{rc} is satisfied with $\kappa=1$ and $\varkappa=2$. It is not hard to check that the transformation described in Theorem~\ref{Th1} with $N=2$ reduces the system to
\begin{gather*}
\begin{split}
& \frac{dr}{dt}=t^{-\frac12}\Lambda_1(r,\psi)+t^{-1}\Lambda_2(r,\psi)+\widetilde \Lambda_2(r,\psi,S(t),t), \\
& \frac{d\psi}{dt}=t^{-1}\Omega_2(r,\psi)+\widetilde \Omega_2(r,\psi,S(t),t),
\end{split}
\end{gather*}
where
\begin{align*}
&\Lambda_1(r,\psi)\equiv \frac{r(4b_0+3 c_0 r^2)}{8}, \quad 
\Lambda_2(r,\psi)\equiv \frac{b_1 r \sin 2\psi}{4}, \\
&\Omega_2(r,\psi)\equiv \frac{1}{256} \left(-32 b_0^2 - 48 b_0 c_0 r^2 - 27 c_0^2 r^4 - 128 s_2 + 
 64 b_1\cos 2 \psi\right),
\end{align*}
and $\widetilde \Lambda_2(r,\psi,S,t)=\mathcal O(t^{-3/2})$, $\widetilde \Omega_2(r,\psi,S,t)=\mathcal O(t^{-3/2})$ as $t\to\infty$ uniformly for all $r\in (0,\mathcal R]$, $(\psi,S)\in\mathbb R^2$ with any $\mathcal R>0$. It follows that assumption \eqref{asnm} holds with $n=1$ and $m=2$.  

If  $b_0 c_0<0$ and $|b_0^2+8s_2|<4|b_1|$, then assumption \eqref{as1} holds with
\begin{gather*}
\varrho_\ast=\sqrt{-\frac{4b_0}{3c_0}}, \quad \varphi_\ast= \pm \theta_\ast+\pi k, \quad k\in\mathbb Z, \quad 
\theta_\ast=\frac{1}{2}\arccos \frac{b_0^2+8s_2}{4b_1}.
\end{gather*}
In this case,
\begin{gather*}
\lambda_n=-b_0, \quad 
\nu_n=0,\quad 
\eta_m=\frac{3 c_0 b_0 \varrho_\ast}{16}, \quad 
\omega_m=-\frac{b_1 \sin 2\varphi_\ast}{2},
\end{gather*}
$\alpha_1^0=\beta_{n,m,1}=\lambda_n$, $\alpha_2^0=\omega_m$, and $\beta_{n,m,2}=\omega_m+1/4$. If, in addition, $b_0>0$ and $\pm b_1 \sin 2\theta_\ast>1/2$, then it follows from Lemma~\ref{Lem2} and Theorem~\ref{Th2} that the persistent phase-locking occurs such that system \eqref{Ex3} has stable regimes with $\varrho(t)\approx \varrho_\ast$, $\varphi(t)\approx  S(t)/2\pm\theta_\ast+\pi k$, $k\in\mathbb Z$ (see~Fig.~\ref{FigEx31}).

\begin{figure}
\centering
{
   \includegraphics[width=0.4\linewidth]{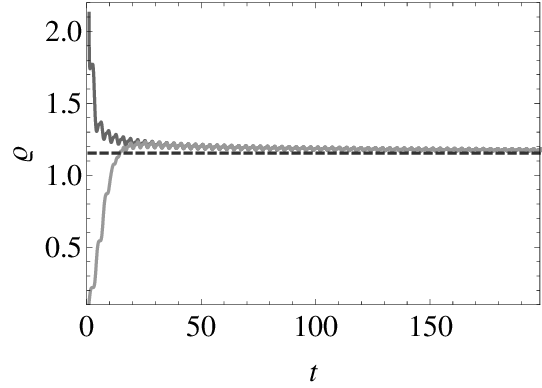}
}
\hspace{1ex}
{
   	\includegraphics[width=0.4\linewidth]{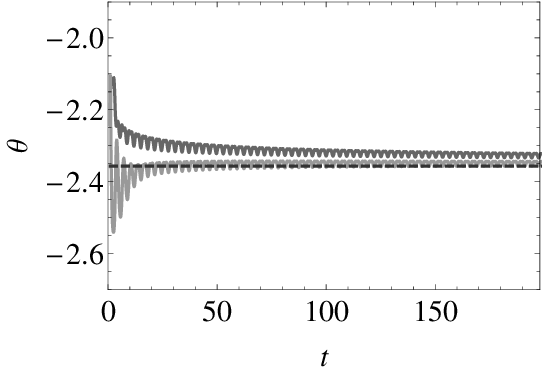}
}
\caption{\small The evolution of $\varrho(t)\equiv \sqrt{x_1^2(t)+x_2^2(t)}$ and $\theta(t)\equiv \varphi(t)-S(t)/2$, $\tan \varphi(t)=-x_2(t)/x_1(t)$ for solutions to system \eqref{Ex3} with $s_2=-1/8$, $b_0=b_1=1$, $c_0=-1$ and different values of initial data. The dashed curves correspond to $\varrho(t)\equiv \varrho_\ast$ and $\theta(t)\equiv \theta_\ast-\pi$, where $\varrho_\ast=2/\sqrt{3}$ and $\theta_0=\pi/4$.} \label{FigEx31}
\end{figure}

If  $b_0 c_0<0$ and $|b_0^2+8s_2|>4|b_1|$, then assumption \eqref{as2} holds with
\begin{gather*}
\varrho_\ast=\sqrt{-\frac{4b_0}{3c_0}}, \quad 
\ell_n(\psi)\equiv -b_0, \quad 
\Omega_2(\varrho_\ast,\psi)\equiv -\frac{1}{16}(b_0^2+8s_2-4b_1\cos 2\psi)\neq 0 \quad \forall \psi\in\mathbb R.
\end{gather*}
 It follows from Theorem~\ref{Th3} that if $b_0>0$, then there is a persistent phase drift regime for solutions of system \eqref{PS} such that $\varrho(t)\approx \varrho_\ast$ and $\varphi-S(t)/2\to \infty$ as $t\to\infty$  (see~Fig.~\ref{FigEx32}). This regime is unstable if $b_0<0$.

\begin{figure}
\centering
{
   \includegraphics[width=0.4\linewidth]{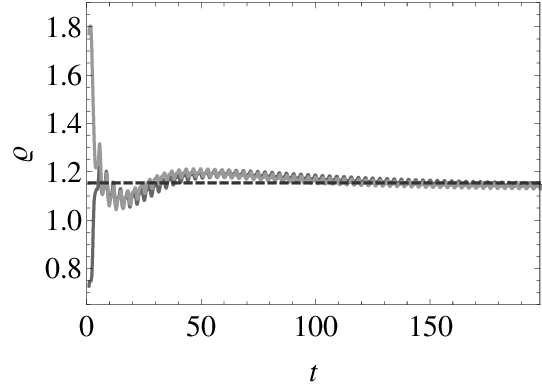}
}
\hspace{1ex}
{
   	\includegraphics[width=0.4\linewidth]{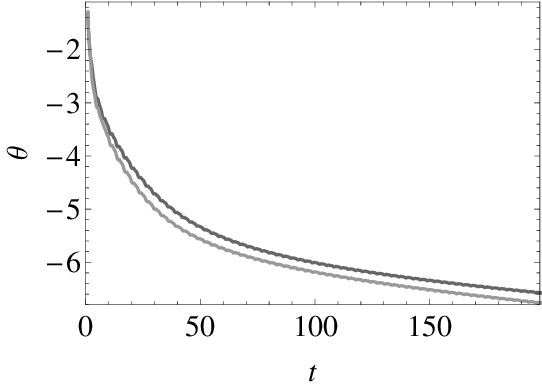}
}
\caption{\small The evolution of $\varrho(t)\equiv \sqrt{x_1^2(t)+x_2^2(t)}$ and $\theta(t)\equiv \varphi(t)-S(t)/2$, $\tan \varphi(t)=-x_2(t)/x_1(t)$ for solutions to system \eqref{Ex3} with $s_2=2$, $b_0=b_1=1$, $c_0=-1$ and different values of initial data. The dashed curve corresponds to $\varrho(t)\equiv \varrho_\ast$, where $\varrho_\ast=2/\sqrt{3}$.} \label{FigEx32}
\end{figure}

\section*{Acknowledgements}
The author thanks Prof. L.A. Kalyakin for useful discussions.

%\section*{Funding}
The research is made in the framework of realization of the development program of Scientific Educational Mathematical Center of Privolzhsky Federal District, agreement 075-02-2024-1444.
% Работа выполнена в рамках реализации программы развития Научно-образовательного математического центра Приволжского федерального округа (соглашение № 075-02-2024-1444)

%\section*{Competing interests}
%The author has no competing interests to declare that are relevant to the content of this article.

}
\end{document}